\documentclass[11pt]{article}

\usepackage[a4paper,margin=1in]{geometry}
\usepackage{amsmath,amssymb,amsfonts,amsthm}
\usepackage{mathtools}
\usepackage{bm}
\usepackage{booktabs}
\usepackage{graphicx}
\usepackage{xcolor}
\usepackage{listings}
\usepackage{hyperref}

\numberwithin{equation}{section}

\newtheorem{assumption}{Assumption}[section]
\newtheorem{lemma}{Lemma}[section]
\newtheorem{theorem}{Theorem}[section]
\newtheorem{remark}{Remark}[section]

\title{A Projected Drift-Randomized Milstein Method for SDEs with Non-differentiable and Super-linear Drift Coefficients}
\author{
	Shuai Wang\\[1mm]
	\small Financial Technology Thrust, Society Hub\\
	\small The Hong Kong University of Science and Technology (Guangzhou)\\[1mm]
	\small \href{mailto:swang767@connect.hkust-gz.edu.cn}
	{swang767@connect.hkust-gz.edu.cn}
}
\date{\today}

\begin{document}
	
	\maketitle
	
\begin{abstract}
	We propose a projected drift-randomized Milstein (PRM) method for stochastic
	differential equations with non-differentiable and super-linearly growing drift
	coefficients. The method extends the randomized Milstein approach beyond the globally Lipschitz setting by incorporating a drift projection into the randomized
	quadrature approximation. Moreover, unlike existing first-order Milstein-type
	methods for SDEs with super-linearly growing drift coefficients, the proposed
	method does not require spatial differentiability of the drift coefficient. Under
	suitable polynomial Lipschitz and one-sided Lipschitz conditions on the drift,
	together with standard regularity assumptions on the diffusion, we establish a
	one-step mean-square stability estimate and derive the required local residual
	bounds. These estimates yield first-order strong convergence of the PRM method
 in the \(L^2\)-sense. Numerical experiments confirm
	the theoretical convergence rate and demonstrate the applicability of the
	method to SDEs with non-differentiable and super-linearly growing drifts.
	
\end{abstract}

	\section{Introduction}
	
Stochastic differential equations (SDEs) provide a fundamental framework for
modeling dynamical systems subject to random fluctuations, with applications
in finance, physics, biology, engineering, and many other fields
\cite{van2007stochastic,Grigoriu}. Since explicit solutions are rarely
available, the construction and analysis of numerical methods for SDEs have
received considerable attention; classical references include
\cite{kloeden1992numerical,milsbook}. For SDEs with globally Lipschitz coefficients, the Euler--Maruyama method is one
of the most basic and widely used numerical schemes. However, many models of
practical interest involve drift or diffusion coefficients with super-linear
growth. In this setting, the classical Euler--Maruyama method may fail to
converge and can even diverge in both the strong and weak senses
\cite{hutzenthaler2011strong}. This phenomenon has motivated the development of
various stabilized Euler-type methods for SDEs with non-globally Lipschitz
coefficients, including tamed Euler methods, truncated Euler methods, projected
Euler methods, and backward Euler methods; see, for example,
\cite{sabanis2013note,mao2015truncated,BeynKru,mao2013strong}. To achieve higher-order convergence, numerical schemes can be constructed from
the It\^o--Taylor expansion. Among them, the Milstein method is one of the most
widely studied and attains strong convergence of order one under suitable
regularity conditions. Stabilization techniques developed for Euler-type schemes have also been extended to Milstein-type methods with locally Lipschitz and super-linearly growing coefficients. Representative examples include the tamed Milstein method \cite{wang2013tamed}, the projected Milstein method
\cite{beyn2017stochastic}, and other related variants.

Randomized methods form a special class of numerical methods inspired by
Monte Carlo simulation. By introducing additional uniformly distributed random
variables into the numerical scheme, such methods may exhibit better
approximation properties than their deterministic counterparts in certain
situations. For instance, randomized Euler methods can be used to handle SDEs
with time-irregular or discontinuous coefficients, for which the classical
Euler algorithm may fail to converge to the exact solution
\cite{przybylowicz2014strong,przybylowicz2014optimality}. Moreover,
randomization techniques have also shown advantages in sampling and machine
learning algorithms; see, for example,
\cite{shen2019randomized,yu2024langevin,wang2025langevin}. Beyond randomized
Euler methods, a higher-order randomized Milstein method was proposed in
\cite{KruWu2019} for approximating SDEs with non-differentiable drift
coefficients. The key ingredient is a randomized quadrature rule for
H\"older-continuous stochastic processes. It was shown that this randomized
Milstein method can achieve a higher convergence rate than the classical
Milstein method in the \(L^p\)-sense. More precisely, when the drift coefficient is globally Lipschitz with respect
to the state variable and only \(1/2\)-H\"older continuous with respect to time,
corresponding to the case \(\kappa=0\) in Assumption~\ref{ass:drift}, the randomized
Milstein method in \cite{KruWu2019} achieves strong convergence of order one,
whereas the classical Milstein method generally attains only order \(1/2\). Besides applications to classical SDEs, \cite{biswas2024explicit}
extends this approach to McKean--Vlasov SDEs with non-differentiable drift
coefficients, \cite{przybylowicz2024randomized} studies randomized Milstein
algorithms for jump--diffusion SDEs, and \cite{morkisz2021randomized}
investigates randomized methods for SDEs under noisy information. However,
the above works are mainly restricted to SDEs with globally Lipschitz drift
coefficients. More recently, \cite{biswas2026randomized} extended the
randomized Milstein method to SDEs with super-linearly growing drift
coefficients by combining taming and randomization techniques. Although this
method can achieve order one convergence for SDEs with drift coefficients
that are \(1/2\)-H\"older continuous in time, it does not preserve another
important advantage of the randomized Milstein method in \cite{KruWu2019},
namely its applicability to non-differentiable drift coefficients, since
additional differentiability assumptions on the drift are required (see Remark~\ref{Remark:assu-compare} for more details).

In this paper, we consider SDEs whose drift coefficient is not necessarily
differentiable and may grow super-linearly in the state variable. Our aim is to
extend the randomized Milstein method of \cite{KruWu2019} to the non-globally
Lipschitz setting while retaining its applicability to non-differentiable drift
coefficients. To this end, we propose a novel projected drift-randomized Milstein (PRM) method. The projection is applied only to the drift coefficient in
order to control its super-linear growth, while the drift-randomized quadrature
rule is used to approximate the drift integral and accommodate its limited
regularity. The super-linear growth of the drift makes the convergence analysis
in \cite{KruWu2019} no longer directly applicable. We therefore develop a new
argument based on uniform moment estimates, one-step mean-square stability, and
local residual estimates. Under suitable polynomial Lipschitz and one-sided
Lipschitz conditions on the drift coefficient, we prove that the PRM method converges strongly with order one at the temporal grid points in the
\(L^2\)-sense. Compared with existing first-order Milstein-type methods for SDEs
with super-linearly growing drift coefficients, the PRM method does not
require spatial differentiability of the drift coefficient. To the best of our
knowledge, this is the first randomized Milstein-type method that achieves
first-order strong convergence for SDEs with drift coefficients that are both
non-differentiable and super-linearly growing. Numerical simulations are
provided to support the theoretical result.
	
The rest of this paper is organized as follows. 
Section~\ref{sec:pre} introduces the basic notation, standing assumptions, and
the projected drift-randomized Milstein scheme. 
Section~\ref{sec:moments} establishes moment bounds for the exact and numerical
solutions, together with several auxiliary estimates needed for the convergence
analysis. 
Section~\ref{sec:convergence} proves the main strong convergence result of the
proposed method. 
Finally, Section~\ref{sec:numerical-experiment} presents numerical experiments
to illustrate the theoretical findings.
\section{Preliminaries}\label{sec:pre}
\subsection{Notation}

Throughout this paper, \(T>0\) is fixed. Let
$
(\Omega_W,\mathcal F^W,(\mathcal F_t^W)_{t\in[0,T]},\mathbb P_W)
$
be a filtered probability space satisfying the usual conditions, on which an
\(m\)-dimensional Brownian motion
$
W=(W^1,\ldots,W^m)
$
is defined.
For \(x,y\in\mathbb R^d\), we denote by
$
\langle x,y\rangle
=
\sum_{i=1}^d x_i y_i
$
the Euclidean inner product and by
$
|x|
=
\sqrt{\langle x,x\rangle}
$
the corresponding Euclidean norm. For matrices, \(|\cdot|\) denotes the
Frobenius norm.

In order to introduce the drift randomization, let
$
(\Omega_\tau,\mathcal F^\tau,\mathbb P_\tau)
$
be another probability space on which a sequence
$
(\tau_j)_{j\ge1}
$
of independent random variables uniformly distributed on \((0,1)\) is defined.
We assume that \((\tau_j)_{j\ge1}\) is independent of the
Brownian motion \(W\). We work on the product probability space
$
(\Omega,\mathcal F,\mathbb P)
:=
(\Omega_W\times\Omega_\tau,
\mathcal F^W\otimes\mathcal F^\tau,
\mathbb P_W\otimes\mathbb P_\tau).
$
All random variables defined on \(\Omega_W\) or \(\Omega_\tau\) are naturally
identified with their extensions to the product space.

We denote by \(\mathbb E\), \(\mathbb E_W\), and \(\mathbb E_\tau\) the
expectations with respect to \(\mathbb P\), \(\mathbb P_W\), and
\(\mathbb P_\tau\), respectively. Thus, for every integrable random variable
\(Y\) on \(\Omega\), Fubini's theorem yields
\[
\mathbb E[Y]
=
\mathbb E_\tau\bigl[\mathbb E_W[Y]\bigr]
=
\mathbb E_W\bigl[\mathbb E_\tau[Y]\bigr].
\]
For \(p\ge1\), the space \(L^p(\Omega)\) consists of all
\(\mathbb R^d\)-valued random variables \(Y\) satisfying $
\mathbb E[|Y|^p]<\infty$, we write
\[
\|Y\|_{L^p(\Omega)}
:=
\left(\mathbb E\left[|Y|^p\right]\right)^{1/p}.
\]
When a random variable depends only on the Brownian motion, we may write
\[
\|Y\|_{L^p(\Omega_W)}
:=
\left(\mathbb E_W\left[|Y|^p\right]\right)^{1/p}.
\]
%All random variables defined on \(\Omega_W\) or \(\Omega_\tau\) are naturally
%identified with their extensions to the product space. For a random variable \(Y:\Omega\to\mathbb R^d\), we denote by
%\[
%\mathbb E[Y]:
%=
%\int_{\Omega}Y(\omega_W,\omega_\tau)\,
%d(\mathbb P_W\otimes\mathbb P_\tau)(\omega_W,\omega_\tau)
%\]
%the expectation on the product space. We also use the partial expectations
%\[
%\mathbb E_W[Y](\omega_\tau)
%:=
%\int_{\Omega_W}
%Y(\omega_W,\omega_\tau)\,d\mathbb P_W(\omega_W),
%\]
%and
%\[
%\mathbb E_\tau[Y](\omega_W)
%:=
%\int_{\Omega_\tau}
%Y(\omega_W,\omega_\tau)\,d\mathbb P_\tau(\omega_\tau).
%\]
%Thus, whenever the integrals are well-defined, Fubini's theorem gives
%\[
%\mathbb E[Y]
%=
%\mathbb E_\tau\bigl[\mathbb E_W[Y]\bigr]
%=
%\mathbb E_W\bigl[\mathbb E_\tau[Y]\bigr].
%\]

Let
$
\mathcal F_j^\tau:=\sigma(\tau_1,\ldots,\tau_j),
$
$
\mathcal F_0^\tau:=\{\emptyset,\Omega_\tau\}.
$
For the temporal grid
$
0=t_0<t_1<\cdots<t_N=T,
$
we define
$
\mathcal F_j^h
:=
\mathcal F_{t_j}^W\otimes\mathcal F_j^\tau,
 j=0,1,\ldots,N.
$
In particular, \(\mathcal F_j^h\) contains the Brownian information up to time
\(t_j\) and the randomization variables up to \(\tau_j\). We shall also use the
notation
\[
\mathbb E_j[\cdot]
:=
\mathbb E[\cdot\mid \mathcal F_j^h],
\qquad j=0,1,\ldots,N,
\]
whenever conditional expectations with respect to the discrete filtration are
needed. For \(0\le s<t\le T\), we use the notation
\[
I_{s,t}^{(r)}
:=
\int_s^t dW^r(u)
=
W^r(t)-W^r(s),
\qquad r=1,\ldots,m,
\]
and
\[
I_{s,t}^{(r_1,r_2)}
:=
\int_s^t\int_s^{u_1}
dW^{r_1}(u_2)\,dW^{r_2}(u_1),
\qquad r_1,r_2=1,\ldots,m.
\]
The generic constant \(C>0\) may change from line to line, but is always
independent of the step size and of the time index. Constants depending on
\(p\) or other parameters are denoted by \(C_p\), \(C_q\), and so on.
	
	In this paper, we consider the SDE
	\begin{equation}
		\label{eq:sde}
		dX(t)
		=
		f(t,X(t))\,dt
		+
		\sum_{r=1}^m g^r(t,X(t))\,dW^r(t),
		\qquad t\in[0,T],
	\end{equation}
	with initial condition
	\begin{equation}
		X(0)=X_0.
	\end{equation}
Here
$
f:[0,T]\times\mathbb R^d\to\mathbb R^d
$
is the drift coefficient, and
$
g^r:[0,T]\times\mathbb R^d\to\mathbb R^d,
r=1,\ldots,m,
$
are the diffusion coefficients. The solution of \eqref{eq:sde}, denoted by
$
X:[0,T]\times\Omega_W\to\mathbb R^d,
$
is adapted to the filtration \((\mathcal F_t^W)_{t\in[0,T]}\).

	For the Milstein coefficient, define,
for \(r_1,r_2=1,\ldots,m\), 
\begin{equation}\label{partial-priduct-gr}
g^{r_1,r_2}(t,x)
=
\partial_x g^{r_1}(t,x)\,g^{r_2}(t,x),	
\end{equation}
where \(\partial_x g^{r_1}(t,x)\) denotes the Jacobian matrix of
\(g^{r_1}\) with respect to the state variable \(x\). Thus
\(g^{r_1,r_2}(t,x)\in\mathbb R^d\). Equivalently, for \(i=1,\ldots,d\),
\[
\bigl(g^{r_1,r_2}(t,x)\bigr)^i
=
\sum_{\ell=1}^d
\frac{\partial g^{r_1,i}}{\partial x_\ell}(t,x)
g^{r_2,\ell}(t,x).
\]
		\subsection{Assumptions}
To complete our convergence analysis, we consider the following assumptions.
	\begin{assumption}[Initial value]
		\label{ass:X0}
		The initial value \(X_0\) is \(\mathcal F_0^W\)-measurable, independent of the
		pair \(\bigl(W,(\tau_j)_{j\ge1}\bigr)\), and satisfies
		\[
		X_0\in L^q(\Omega)
		\qquad\text{for every }q\ge1.
		\]
	\end{assumption}
	
	\begin{assumption}[Drift coefficient]
		\label{ass:drift}
		There exist constants \(K>0\) and \(\kappa\geq0\) such that, for all \(t\in[0,T]\) and \(x,y\in\mathbb R^d\)
		\begin{equation}
			\label{eq:drift-poly-lip}
			|f(t,x)-f(t,y)|
			\le
			K(1+|x|^\kappa+|y|^\kappa)|x-y|.
		\end{equation}
	 Moreover,
\begin{equation}\label{eq:drift-0}
		\sup_{t\in[0,T]}|f(t,0)|<\infty.
\end{equation}
		Furthermore, \(f\) is \(1/2\)-H\"older continuous in time:
		\begin{equation}
			\label{eq:drift-time-holder}
			|f(t,x)-f(s,x)|
			\le
			K(1+|x|^{\kappa+1})|t-s|^{1/2}.
		\end{equation}
		Finally, \(f\) satisfies the one-sided Lipschitz condition
		\begin{equation}
			\label{eq:drift-one-sided}
			\langle x-y,f(t,x)-f(t,y)\rangle
			\le
			K|x-y|^2.
		\end{equation}
	\end{assumption}
	\begin{assumption}[Diffusion coefficient]
		\label{ass:g}
		For each \(r=1,\dots,m\) and every fixed
		\(t\in[0,T]\), the mapping \(x\mapsto g^r(t,x)\) is continuously differentiable.
		Moreover, there exists a constant \(K>0\) such that, for all \(t\in[0,T]\) and \(x,y\in\mathbb R^d\)
\begin{equation}\label{eq:g0}
		\sup_{t\in[0,T]}|g^r(t,0)|<\infty.
\end{equation}
		\begin{equation}
			\label{eq:g-lip}
			|g^r(t,x)-g^r(t,y)|
			\le
			K|x-y|,
		\end{equation}
		\begin{equation}
			\label{eq:gx-bounded}
			|\partial_x g^r(t,x)|\le K,
		\end{equation}
		and
		\begin{equation}
			\label{eq:gx-lip}
			|\partial_x g^r(t,x)-\partial_x g^r(t,y)|
			\le
			K|x-y|
		\end{equation}
Furthermore, we assume
		\begin{equation}
			\label{eq:g12-lip}
			|g^{r_1,r_2}(t,x)-g^{r_1,r_2}(t,y)|
			\le
			K|x-y|.
		\end{equation}
 Finally, \(g^r\) is Lipschitz continuous in time:
		\begin{equation}
			\label{eq:g-time-lip}
			|g^r(t,x)-g^r(s,x)|
			\le
			K(1+|x|)|t-s|,
		\end{equation}
	\end{assumption}
%	\begin{remark}
%		The assumptions imposed here are closely related to those in the randomized
%		Milstein literature. Compared with the classical randomized Milstein method
%		\cite{BeynKru}, we replace the global Lipschitz condition on the drift
%		coefficient by a polynomial growth condition and a one-sided Lipschitz
%		condition, which allows for super-linear drift growth. Compared with the recent
%		randomized-tamed Milstein method \cite{biswas2026randomized}, our framework does
%		not require differentiability of the drift coefficient. In particular, no
%		condition of the form
%		\[
%		|\partial_x^2 f(t,x)|\le C(1+|x|^{r-1})
%		\]
%		is imposed. Hence the present method can be applied to SDEs with
%		non-differentiable drift coefficients. The time regularity assumptions are
%		chosen as \(1/2\)-H\"older continuity for the drift coefficient and Lipschitz
%		continuity for the diffusion coefficient, corresponding to the case
%		\(\gamma=1/2\) in \cite{BeynKru}, which is sufficient for the first-order
%		convergence result proved below.
%	\end{remark}
	\begin{remark}\label{Remark:assu-compare}
Compared with the classical randomized Milstein method in \cite{BeynKru}, the
global Lipschitz condition on the drift coefficient is replaced in this paper by
a polynomial growth condition together with a one-sided Lipschitz condition.
This allows us to treat drift coefficients with super-linear growth in the
state variable. On the other hand, the recent work \cite{biswas2026randomized}
requires an additional differentiability assumption on the drift coefficient,
for instance
$
|\partial_x^2 f(t,x)|\le C(1+|x|^{\kappa-1}),
$
which prevents its direct application to SDEs with non-differentiable drift
coefficients. This observation is one of the main motivations of the present
work. The remaining assumptions are in line with those in
\cite{BeynKru,biswas2026randomized}. In particular, we impose
\(1/2\)-H\"older continuity in time on the drift coefficient and Lipschitz
continuity in time on the diffusion coefficient, which corresponds to the case
\(\gamma=1/2\) in \cite{BeynKru}, in order to obtain the first-order convergence
rate.
	\end{remark}
	\begin{remark}
		\label{rem:one-sided-growth}
		The one-sided Lipschitz condition on the drift implies a useful coercivity-type
		estimate. Indeed, taking \(y=0\) in \eqref{eq:drift-one-sided} gives
$$
			\langle x,f(t,x)-f(t,0)\rangle
			\le
			K|x|^2 .
$$
		Hence
		\begin{equation}
			\label{eq:rem-xf-split}
		\langle	x,f(t,x)\rangle
			\le
			K|x|^2+\langle x ,f(t,0)\rangle.
		\end{equation}
 An application of \eqref{eq:drift-0} and Young's inequality gives
		\begin{equation}
			\label{eq:rem-xf-zero-bound}
			\langle x ,f(t,0)\rangle
			\le
			|x|\,|f(t,0)|
			\le
			C|x|
			\le
			C(1+|x|^2).
		\end{equation}
		Combining \eqref{eq:rem-xf-split} and \eqref{eq:rem-xf-zero-bound}, we obtain
		\begin{equation}
			\label{eq:rem-one-sided-drift-growth}
		\langle x ,f(t,x)\rangle
			\le
			C(1+|x|^2),
			\qquad t\in[0,T],\ x\in\mathbb R^d .
		\end{equation}
Moreover, by Assumptions~\ref{ass:drift} and \ref{ass:g}, we have the growth conditions for $f,g^r$
\begin{equation}
	\label{eq:drift-growth}
	|f(t,x)|\le C(1+|x|^{\kappa+1}).
\end{equation}	
\begin{equation}
	\label{eq:diffusion-growth}
	|g^r(t,x)|\le C(1+|x|), \quad r=1,\dots,m.
\end{equation}	
%		Moreover, the usual monotonicity inequality involving the
%		diffusion coefficient follows from \eqref{eq:drift-one-sided} and the global
%		Lipschitz continuity of \(g^r\) \eqref{eq:g-lip}. Indeed, for every fixed \(p\ge1\), we have
%		\begin{align}
%			\label{eq:rem-derived-monotonicity}
%			&\langle x-y,f(t,x)-f(t,y)\rangle
%			+
%			\frac{2p-1}{2}
%			\sum_{r=1}^m
%			|g^r(t,x)-g^r(t,y)|^2
%			\nonumber\\
%			&\quad\le
%			K|x-y|^2
%			+
%			\frac{2p-1}{2}mK^2|x-y|^2
%			\nonumber\\
%			&\quad\le
%			C_p|x-y|^2 .
%		\end{align}
%		Here \(C_p>0\) may depend on \(p\), but it is independent of \(t,x,y\) and
%		the step size. Thus the monotonicity condition does not need
%		to be imposed as a separate assumption.
	\end{remark}
	\begin{remark}
		If the diffusion coefficient is autonomous, namely \(g^r(t,x)=g^r(x)\), then the time regularity assumption \eqref{eq:g-time-lip} is automatically satisfied. This is the case in the numerical experiments in Section~\ref{sec:numerical-experiment}.
	\end{remark}
	\subsection{Projected drift-randomized Milstein method}
	
Recall the temporal grid
$
	0=t_0<t_1<\cdots<t_N=T
$
with step sizes
$
	h_j=t_j-t_{j-1},
	h=\max_{1\le j\le N}h_j.
$
	Throughout the analysis, we consider temporal grids satisfying
	\begin{equation}
		\label{eq:mesh-size-condition}
		0<h\le1.
	\end{equation}
	This entails \(0<h_j\le1\) for all \(j=1,\ldots,N\). Define the projection operator
	\begin{equation}
		\label{eq:projection}
		T_h(x)
		=
		x\min\left(1,\frac{h^{-\alpha}}{|x|}\right),
	\end{equation}
	with the convention that \(T_h(0)=0\) and 	
	\begin{equation}
		\label{eq:alpha}
		0<\alpha<\frac{1}{2(\kappa+1)}.
	\end{equation}
	Thus,
	\[
	T_h(x)=x
	\quad\text{if } |x|\le h^{-\alpha},
	\]
	and
	\[
	|T_h(x)|\le h^{-\alpha}.
	\]
	We define the projected drift
	\begin{equation}
		\label{eq:fh}
		f_h(t,x)=f(t,T_h(x)).
	\end{equation}
The proposed PRM method is defined as follows.
	Given \(X_h^{j-1}\), first compute the randomized internal stage
	\begin{equation}
		\label{eq:Xtau}
		X_h^{j,\tau}
		=
		X_h^{j-1}
		+
		\tau_jh_j f_{h_j}(t_{j-1},X_h^{j-1})
		+
		\sum_{r=1}^m
		g^r(t_{j-1},X_h^{j-1})
		I_{t_{j-1},\,t_{j-1}+\tau_jh_j}^{(r)}.
	\end{equation}
	Then define
	\begin{align}
		\label{eq:scheme}
		X_h^j
		=
		X_h^{j-1}
		&+
		h_j f_{h_j}
		\bigl(t_{j-1}+\tau_jh_j,X_h^{j,\tau}\bigr)+
		\sum_{r=1}^m
		g^r(t_{j-1},X_h^{j-1})
		I_{t_{j-1},t_j}^{(r)}
		\nonumber\\
		&+
		\sum_{r_1,r_2=1}^m
		g^{r_1,r_2}(t_{j-1},X_h^{j-1})
		I_{t_{j-1},t_j}^{(r_2,r_1)}.
	\end{align}
	The initial value is
$
	X_h^0=X_0.
$
	
%	\begin{remark}
%		The projection is used only inside the drift coefficient \(f_h=f\circ T_h\). The base point of the method remains \(X_h^{j-1}\). The diffusion and Milstein correction are evaluated at the unprojected point \(X_h^{j-1}\). This is essential for preserving the It\^o--Taylor structure of the Milstein method.
%	\end{remark}
	
	\section{Moment Bounds and Auxiliary Estimates}\label{sec:moments}
	
	\begin{lemma}[Moment bounds and temporal regularity of the exact solution]
		\label{lem:exact-moment}
		Let Assumptions~\ref{ass:X0}--\ref{ass:g} hold. Then the SDE
		\eqref{eq:sde} admits a unique solution \(X\). Moreover, for every \(q\ge2\),
		there exists a constant \(C_q>0\) such that
		\begin{equation}
			\label{eq:exact-moment-bound}
			\sup_{0\le t\le T}\mathbb E\left[|X(t)|^q\right]\le C_q .
		\end{equation}
		Furthermore, for every \(q\ge2\),
		\begin{equation}
			\label{eq:exact-holder-bound}
			\|X(t)-X(s)\|_{L^q(\Omega_W)}
			\le
			C_q |t-s|^{1/2},
			\qquad s,t\in[0,T].
		\end{equation}
	\end{lemma}
	
	\begin{proof}
		Existence and uniqueness follow from the
		polynomial growth condition \eqref{eq:drift-poly-lip}, the one-sided Lipschitz condition
		\eqref{eq:drift-one-sided}, and the global Lipschitz of the
		diffusion coefficient \eqref{eq:g-lip}.
		
		We prove the moment estimate. Fix \(q\ge2\). By It\^o's formula,
		\begin{align}
			\label{eq:exact-ito-q}
	|X(t)|^q
	&=
	|X_0|^q
	+
	q\int_0^t
	|X(s)|^{q-2}
	\langle X(s),f(s,X(s))\rangle\,ds
	\nonumber\\
	&\quad+
	\frac q2
	\sum_{r=1}^m
	\int_0^t
	|X(s)|^{q-2}|g^r(s,X(s))|^2\,ds
	\nonumber\\
	&\quad+
	\frac{q(q-2)}2
	\sum_{r=1}^m
	\int_0^t
	|X(s)|^{q-4}
	\left|
	\left\langle X(s),g^r(s,X(s))\right\rangle
	\right|^2\,ds
	\nonumber\\
	&\quad+
	q\sum_{r=1}^m
	\int_0^t
	|X(s)|^{q-2}
	\left\langle X(s),g^r(s,X(s))\right\rangle
	\,dW^r(s).
\end{align}
		By \eqref{eq:rem-one-sided-drift-growth},
		\[
		\langle X(s),f(s,X(s))\rangle\le C(1+|X(s)|^2).
		\]
		Moreover, by \eqref{eq:diffusion-growth}, we have
		\[
		|g^r(s,x)|^2\le C(1+|x|^2).
		\]
		Therefore, after taking expectations in \eqref{eq:exact-ito-q}, the martingale
		term vanishes and we obtain
		\[
		\mathbb E\left[|X(t)|^q\right]
		\le
		\mathbb E\left[|X_0|^q\right]
		+
		C_q\int_0^t
		\bigl(1+\mathbb E\left[|X(s)|^q\right]\bigr)\,ds .
		\]
	A standard stopping time truncation and	Gronwall's inequality gives \eqref{eq:exact-moment-bound}.
		
		For temporal regularity, let \(0\le s<t\le T\). From the integral form of
		\eqref{eq:sde},
		\[
		X(t)-X(s)
		=
		\int_s^t f(u,X(u))\,du
		+
		\sum_{r=1}^m
		\int_s^t g^r(u,X(u))\,dW^r(u).
		\]
		Using the growth bound \eqref{eq:drift-growth}, H\"older's inequality, and
		\eqref{eq:exact-moment-bound} with moment order \(q(\kappa+1)\), we obtain
		\[
		\left\|
		\int_s^t f(u,X(u))\,du
		\right\|_{L^q(\Omega_W)}
		\le
		\int_s^t
		\|f(u,X(u))\|_{L^q(\Omega_W)}\,du
		\le
		C_q|t-s|.
		\]
		By the Burkholder--Davis--Gundy inequality and the linear growth of \(g^r\),
		\[
		\left\|
		\sum_{r=1}^m
		\int_s^t g^r(u,X(u))\,dW^r(u)
		\right\|_{L^q(\Omega_W)}
		\le
		C_q
		\left(
		\int_s^t
		\bigl(1+\|X(u)\|_{L^q(\Omega_W)}^2\bigr)\,du
		\right)^{1/2}
		\le
		C_q|t-s|^{1/2}.
		\]
		Combining the last two estimates yields \eqref{eq:exact-holder-bound}.
	\end{proof}
	
	\begin{lemma}[Projection properties and errors]
		\label{lem:projection-error}
 For the projection operator
		\(T_h\) defined by \eqref{eq:projection}, the following properties hold for all
		\(x,y\in\mathbb R^d\):
		\[
		|T_h(x)|\le |x|,
		\qquad
		|T_h(x)|\le h^{-\alpha},
		\]
		and the non-expansiveness property
		\begin{equation}\label{eq:Th-nonexpansive}
			|T_h(x)-T_h(y)|\le |x-y|.
		\end{equation}
		
		Moreover, let \(\ell>2(\kappa+1)\) and let
		\(Z\in L^\ell(\Omega)\). Then there exists a constant
		\(C_\ell>0\), depending on \(\ell\), \(\|Z\|_{L^\ell(\Omega)}\),
		and the coefficients, but independent of \(h\) and \(t\), such that
		\begin{equation}
			\label{eq:projection-state-error-new}
			\|Z-T_h(Z)\|_{L^2(\Omega)}
			\le
			C_\ell h^{\frac{\alpha}{2}(\ell-2)}.
		\end{equation}
		Furthermore, for all \(t\in[0,T]\),
		\begin{equation}
			\label{eq:projection-drift-error-new}
			\|f(t,Z)-f_h(t,Z)\|_{L^2(\Omega)}
			\le
			C_\ell h^{\frac{\alpha}{2}(\ell-2(\kappa+1))}.
		\end{equation}
		In particular, if
		\begin{equation}
			\label{eq:ell-large-enough-new}
			\ell\ge 2(\kappa+1)+\frac{2}{\alpha},
		\end{equation}
		then
		\begin{equation}
			\label{eq:projection-drift-error-order-one-new}
			\|f(t,Z)-f_h(t,Z)\|_{L^2(\Omega)}
			\le
			C_\ell h.
		\end{equation}
	\end{lemma}
	
	\begin{proof}
		Let
		\[
		R_h:=h^{-\alpha}.
		\]
		By the definition of \(T_h\), we have
		\[
		T_h(x)=x, \qquad \text{if } |x|\le R_h,
		\]
		and
		\[
		T_h(x)=R_h\frac{x}{|x|}, \qquad \text{if } |x|>R_h.
		\]
		Therefore,
		\[
		|T_h(x)|\le |x|,
		\qquad
		|T_h(x)|\le R_h=h^{-\alpha}.
		\]
		
We next prove the non-expansiveness of \(T_h\). We distinguish three cases.

If \(|x|\le R_h\) and \(|y|\le R_h\), then \(T_h(x)=x\) and \(T_h(y)=y\).
Hence
\[
|T_h(x)-T_h(y)|=|x-y|.
\]

If \(|x|\le R_h<|y|\), then \(T_h(x)=x\) and
\[
T_h(y)=R_h\frac{y}{|y|}.
\]
Therefore,
\[
\begin{aligned}
	&|x-y|^2-
	\left|
	x-R_h\frac{y}{|y|}
	\right|^2
	\\
	& =
	|y|^2-R_h^2
	-
	2(|y|-R_h)
	\left\langle x,\frac{y}{|y|}\right\rangle
	\\
	& =
	(|y|-R_h)
	\left(
	|y|+R_h
	-
	2\left\langle x,\frac{y}{|y|}\right\rangle
	\right)
	\\
	& \ge
	(|y|-R_h)
	\left(
	|y|+R_h-2|x|
	\right)
	\\
	& \ge 0.
\end{aligned}
\]
Thus
\[
|T_h(x)-T_h(y)|
=
\left|
x-R_h\frac{y}{|y|}
\right|
\le
|x-y|.
\]
The case \(|y|\le R_h<|x|\) is identical.

Finally, suppose that \(|x|>R_h\) and \(|y|>R_h\). Then
\[
T_h(x)=R_h\frac{x}{|x|},
\qquad
T_h(y)=R_h\frac{y}{|y|}.
\]
Hence
\[
\begin{aligned}
	&|x-y|^2-
	\left|
	R_h\frac{x}{|x|}
	-
	R_h\frac{y}{|y|}
	\right|^2
	\\
	& =
	|x|^2+|y|^2-2\langle x,y\rangle
	-
	R_h^2
	\left(
	2-2\left\langle\frac{x}{|x|},\frac{y}{|y|}\right\rangle
	\right)
	\\
	& =
	\bigl(|x|-|y|\bigr)^2
	+
	2\left(|x||y|-R_h^2\right)
	\left(
	1-\left\langle\frac{x}{|x|},\frac{y}{|y|}\right\rangle
	\right)
	\\
	& \ge 0.
\end{aligned}
\]
Therefore,
\[
|T_h(x)-T_h(y)|\le |x-y|.
\]
Combining the above cases gives
\[
|T_h(x)-T_h(y)|\le |x-y|,
\]
which proves \eqref{eq:Th-nonexpansive}.
		
		We now prove the projection error estimates. Since \(T_h(Z)=Z\) on
		\(\{|Z|\le h^{-\alpha}\}\), we have
		\begin{equation}
			\label{eq:proj-diff-pointwise}
			|Z-T_h(Z)|
			\le
			|Z|\mathbf 1_{\{|Z|>h^{-\alpha}\}}.
		\end{equation}
		Indeed, on the event \(\{|Z|>h^{-\alpha}\}\),
		\[
		T_h(Z)=h^{-\alpha}\frac{Z}{|Z|},
		\qquad
		|Z-T_h(Z)|=|Z|-h^{-\alpha}\le |Z|.
		\]
		Therefore,
		\begin{align}
			\mathbb E\left[|Z-T_h(Z)|^2\right]
			&\le
			\mathbb E\left[
			|Z|^2\mathbf 1_{\{|Z|>h^{-\alpha}\}}
			\right]
			\nonumber\\
			&\le
			h^{\alpha(\ell-2)}\mathbb E\left[|Z|^\ell \right].
			\nonumber
		\end{align}
		Taking square roots gives \eqref{eq:projection-state-error-new}.
		
		Next, by \eqref{eq:drift-poly-lip},
		\[
		|f(t,Z)-f(t,T_h(Z))|
		\le
		C(1+|Z|^\kappa+|T_h(Z)|^\kappa)|Z-T_h(Z)|.
		\]
		Since \(|T_h(Z)|\le |Z|\), by \eqref{eq:proj-diff-pointwise},
		\[
		|f(t,Z)-f_h(t,Z)|
		\le
		C(1+|Z|^\kappa)|Z|\mathbf 1_{\{|Z|>h^{-\alpha}\}}.
		\]
		Since \(0<h\le1\), we have \(h^{-\alpha}\ge1\). Hence, on
		\(\{|Z|>h^{-\alpha}\}\),
		\[
		(1+|Z|^\kappa)|Z|
		\le
		C|Z|^{\kappa+1}.
		\]
		Thus
		\[
		|f(t,Z)-f_h(t,Z)|
		\le
		C|Z|^{\kappa+1}\mathbf 1_{\{|Z|>h^{-\alpha}\}}.
		\]
		Consequently,
		\begin{align}
			\|f(t,Z)-f_h(t,Z)\|_{L^2(\Omega)}^2
			&\le
			C\mathbb E\left[
			|Z|^{2(\kappa+1)}
			\mathbf 1_{\{|Z|>h^{-\alpha}\}}
			\right]
			\nonumber\\
			&\le
			C h^{\alpha(\ell-2(\kappa+1))}
			\mathbb E|Z|^\ell .
			\nonumber
		\end{align}
		Here we used \(\ell>2(\kappa+1)\). Taking square roots proves
		\eqref{eq:projection-drift-error-new}.
		
		Finally, if \eqref{eq:ell-large-enough-new} holds, then
		\[
		\frac{\alpha}{2}\bigl(\ell-2(\kappa+1)\bigr)\ge1.
		\]
		Therefore, since \(0<h\le1\),
		\[
		h^{\frac{\alpha}{2}(\ell-2(\kappa+1))}
		\le h,
		\]
		and hence \eqref{eq:projection-drift-error-order-one-new} follows.
	\end{proof}
	
	\begin{lemma}[Moment bounds of the numerical solution]
		\label{lem:numerical-moment}
		Let \(X_h^j\) be generated by the projected drift-randomized Milstein method. Then, for every \(p\ge1\), there exists a constant \(C_p>0\), independent of
		\(h\) and \(j\), such that
		\begin{equation}
			\label{eq:lem43-moment-bound}
			\sup_{0\le j\le N}\mathbb E\left[|X_h^j|^{2p}\right]\le C_p.
		\end{equation}
	\end{lemma}
	
	\begin{proof}
Fix \(p\ge1\). We first estimate the internal stage \eqref{eq:Xtau}.
Recall that
\[
\mathcal F_j^h
:=
\mathcal F_{t_j}^W\otimes\mathcal F_j^\tau,
\qquad
\mathbb E_j[\cdot]
:=
\mathbb E[\cdot\mid \mathcal F_j^h],
\qquad j=0,1,\ldots,N .
\]
At the \(j\)-th step, set
\[
Y_j:=X_h^{j-1},
\qquad
V_j:=X_h^{j,\tau}.
\]
Then \(Y_j\) is \(\mathcal F_{j-1}^h\)-measurable and is independent of
\(\tau_j\). Let
\[
a:=\alpha(\kappa+1).
\]
Since \(0<\alpha<1/(2(\kappa+1))\), we have
\[
0<a<\frac12.
\]
We also use the following bound for the projected drift. Since
\(|T_{h_j}(x)|\le h_j^{-\alpha}\), the polynomial growth condition
\eqref{eq:drift-growth} gives
\begin{align}
	\label{eq:lem43-fh-bound-raw}
	|f_{h_j}(t,x)|
	&=
	|f(t,T_{h_j}(x))|
\le
	C\left(1+|T_{h_j}(x)|^{\kappa+1}\right)
\le
	C\left(1+h_j^{-\alpha(\kappa+1)}\right)
	=
	C(1+h_j^{-a}).
\end{align}
Since \(0<h_j\le1\), we have \(h_j^{-a}\ge1\). Hence
\begin{equation}
	\label{eq:lem43-fh-bound}
	|f_{h_j}(t,x)|
	\le
	C h_j^{-a}.
\end{equation}
From the definition of the internal stage \eqref{eq:Xtau}, we have
\begin{equation}
	\label{eq:lem43-stage-decomp}
	V_j-Y_j=\mathcal D_j+\mathcal B_j,
\end{equation}
where
\[
\mathcal D_j:=\tau_jh_j f_{h_j}(t_{j-1},Y_j),
\]
and
\[
\mathcal B_j:=
\sum_{k=1}^m
g^k(t_{j-1},Y_j)
I_{t_{j-1},t_{j-1}+\tau_jh_j}^{(k)}.
\]
By \eqref{eq:lem43-fh-bound} and \(0\le\tau_j\le1\),
\begin{equation}
	\label{eq:lem43-Dj-first}
	|\mathcal D_j|
	\le
	C h_j^{1-a}.
\end{equation}
Consequently,
\begin{equation}
	\label{eq:lem43-Dj-first-cond}
	\mathbb E_{j-1}\left[|\mathcal D_j|\right]
	\le
	C h_j^{1-a},
	\qquad
	\mathbb E_{j-1}\left[|\mathcal D_j|^2\right]
	\le
	C h_j^{2-2a}.
\end{equation}
Next we estimate \(\mathcal B_j\). Define
\[
\mathcal H_j
:=
\mathcal F_{j-1}^h\vee\sigma(\tau_j).
\]
Since \(Y_j\) is \(\mathcal F_{j-1}^h\)-measurable, both \(Y_j\) and
\(\tau_j\) are known under the conditioning with respect to \(\mathcal H_j\).
Moreover, conditionally on \(\mathcal H_j\), the Brownian increments
\[
I_{t_{j-1},t_{j-1}+\tau_jh_j}^{(k)}
=
W^k(t_{j-1}+\tau_jh_j)-W^k(t_{j-1})
\]
are centered and satisfy
\[
\mathbb E\left[
I_{t_{j-1},t_{j-1}+\tau_jh_j}^{(k)}
I_{t_{j-1},t_{j-1}+\tau_jh_j}^{(\ell)}
\,\middle|\,
\mathcal H_j
\right]
=
\delta_{k\ell}\tau_jh_j .
\]
where \(\delta_{rk}\) denotes the Kronecker delta. Therefore, using the independence of the Brownian components and the linear
growth of \(g^k\), we obtain
\begin{align}
	\mathbb E\left[
	|\mathcal B_j|^2
	\,\middle|\,
	\mathcal H_j
	\right]
	&=
	\mathbb E\left[
	\left|
	\sum_{k=1}^m
	g^k(t_{j-1},Y_j)
	I_{t_{j-1},t_{j-1}+\tau_jh_j}^{(k)}
	\right|^2
	\,\middle|\,
	\mathcal H_j
	\right]
	\nonumber\\
	&=
	\tau_jh_j
	\sum_{k=1}^m
	|g^k(t_{j-1},Y_j)|^2
	\nonumber\\
	&\le
	C\tau_jh_j(1+|Y_j|^2)
	\nonumber\\
	&\le
	Ch_j(1+|Y_j|^2)\nonumber.
\end{align}
By Jensen's inequality,
\begin{align}
	\mathbb E\left[
	|\mathcal B_j|
	\,\middle|\,
	\mathcal H_j
	\right]
	&\le
	\left(
	\mathbb E\left[
	|\mathcal B_j|^2
	\,\middle|\,
	\mathcal H_j
	\right]
	\right)^{1/2}\le
	C h_j^{1/2}(1+|Y_j|)\nonumber.
\end{align}
Using the tower property of conditional expectation and the inclusion
\(\mathcal F_{j-1}^h\subset\mathcal H_j\), we get
\begin{align}
	\label{eq:lem43-Bj-first-product}
	\mathbb E_{j-1}\left[|\mathcal B_j|\right]
	&=
	\mathbb E\left[
	\mathbb E\left[
	|\mathcal B_j|
	\,\middle|\,
	\mathcal H_j
	\right]
	\,\middle|\,
	\mathcal F_{j-1}^h
	\right]\le
	C h_j^{1/2}(1+|Y_j|),
\end{align}
and similarly,
\begin{align}
	\label{eq:lem43-Bj-second-product}
	\mathbb E_{j-1}\left[|\mathcal B_j|^2\right]
	&=
	\mathbb E\left[
	\mathbb E\left[
	|\mathcal B_j|^2
	\,\middle|\,
	\mathcal H_j
	\right]
	\,\middle|\,
	\mathcal F_{j-1}^h
	\right]\le
	Ch_j(1+|Y_j|^2).
\end{align}
Combining \eqref{eq:lem43-stage-decomp},
\eqref{eq:lem43-Dj-first-cond}, and
\eqref{eq:lem43-Bj-first-product}, we obtain
\begin{align}
	\label{eq:lem43-stage-increment-first}
	\mathbb E_{j-1}\left[|V_j-Y_j|\right]
	&\le
	\mathbb E_{j-1}\left[|\mathcal D_j|\right]
	+
	\mathbb E_{j-1}\left[|\mathcal B_j|\right]
	\nonumber\\
	&\le
	C h_j^{1-a}
	+
	C h_j^{1/2}(1+|Y_j|).
\end{align}
We next prove the second-moment estimate for the internal stage. From
\eqref{eq:lem43-stage-decomp}, we have
\[
|V_j|^2
\le
3|Y_j|^2+3|\mathcal D_j|^2+3|\mathcal B_j|^2.
\]
Therefore, by \eqref{eq:lem43-Dj-first-cond} and
\eqref{eq:lem43-Bj-second-product},
\begin{align}
	\mathbb E_{j-1}\left[|V_j|^2\right]
	&\le
	C|Y_j|^2
	+
	C h_j^{2-2a}
	+
	C h_j(1+|Y_j|^2)\nonumber.
\end{align}
Since \(a<1/2\) and \(h_j\le1\), we have
\[
h_j^{2-2a}\le C.
\]
Hence
\begin{equation}
	\label{eq:lem43-stage-second-final}
	\mathbb E_{j-1}\left[|V_j|^2\right]
	\le
	C(1+|Y_j|^2).
\end{equation}
We claim that
\begin{equation}
	\label{eq:lem43-one-sided-fh}
	\langle x,f_{h_j}(t,x)\rangle\le C(1+|x|^2),
	\qquad x\in\mathbb R^d.
\end{equation}
Indeed, if \(|x|\le h_j^{-\alpha}\), then \(T_{h_j}(x)=x\), and the claim
follows directly from \eqref{eq:rem-one-sided-drift-growth}. If
\(|x|>h_j^{-\alpha}\), set \(z=T_{h_j}(x)\). Then
\[
z=h_j^{-\alpha}\frac{x}{|x|},
\qquad
x=\frac{|x|}{h_j^{-\alpha}}z.
\]
Hence, by \eqref{eq:rem-one-sided-drift-growth},
\[
\langle x,f(t,T_{h_j}(x))\rangle
=
\frac{|x|}{h_j^{-\alpha}}
\langle z,f(t,z)\rangle
\le
C\frac{|x|}{h_j^{-\alpha}}(1+|z|^2).
\]
Since \(|z|=h_j^{-\alpha}\), we have
\[
\frac{|x|}{h_j^{-\alpha}}(1+|z|^2)
=
\frac{|x|}{h_j^{-\alpha}}
+
|x|h_j^{-\alpha}.
\]
As \(h_j\le1\), \(h_j^{-\alpha}\ge1\), and therefore
\[
\frac{|x|}{h_j^{-\alpha}}
\le
|x|
\le
1+|x|^2.
\]
Moreover, in the present case \(|x|>h_j^{-\alpha}\), which implies
\[
|x|h_j^{-\alpha}
\le
|x|^2.
\]
Consequently,
\[
\langle x,f_{h_j}(t,x)\rangle
\le
C(1+|x|^2).
\]
This proves \eqref{eq:lem43-one-sided-fh}. We now estimate the full step. Since \(Y_j=X_h^{j-1}\) is
\(\mathcal F_{j-1}^h\)-measurable, we have
\begin{equation}
	\label{eq:lem43-Y-cond-exp}
	\mathbb E_{j-1}\left[|Y_j|^{2p}\right]
	=
	|Y_j|^{2p}.
\end{equation}
Write the full step as
$$
	X_h^j=Y_j+\Delta_j,
$$
where
\begin{equation}
	\label{eq:lem43-delta-def}
	\Delta_j
	=
	h_jF_j+G_j+M_j,
\end{equation}
with
\[
F_j:=f_{h_j}(t_{j,\tau},V_j),
\qquad
t_{j,\tau}:=t_{j-1}+\tau_jh_j,
\]
\[
G_j
:=
\sum_{r=1}^m
g^r(t_{j-1},Y_j)
I_{t_{j-1},t_j}^{(r)},
\]
and
\[
M_j
:=
\sum_{r_1,r_2=1}^m
g^{r_1,r_2}(t_{j-1},Y_j)
I_{t_{j-1},t_j}^{(r_2,r_1)}.
\]
We use the elementary Taylor-type inequality
\begin{equation}
	\label{eq:lem43-taylor-ineq}
	|x+y|^{2p}
	\le
	|x|^{2p}
	+
	2p|x|^{2p-2}\langle x,y\rangle
	+
	C_p\left(
	|x|^{2p-2}|y|^2+|y|^{2p}
	\right),
	\qquad x,y\in\mathbb R^d.
\end{equation}
Applying \eqref{eq:lem43-taylor-ineq} with \(x=Y_j\) and \(y=\Delta_j\), we get
\begin{align}
	|X_h^j|^{2p}
	&\le
	|Y_j|^{2p}
	+
	2p|Y_j|^{2p-2}\langle Y_j,\Delta_j\rangle
	+
	C_p|Y_j|^{2p-2}|\Delta_j|^2
	+
	C_p|\Delta_j|^{2p}\nonumber.
\end{align}
Taking the conditional expectation \(\mathbb E_{j-1}[\cdot]\) and using
\eqref{eq:lem43-Y-cond-exp}, we obtain
\begin{align}
	\label{eq:lem43-after-taylor}
	\mathbb E_{j-1}\left[
	|X_h^j|^{2p}
	\right]
	&\le
	|Y_j|^{2p}
	+
	2p\,
	\mathbb E_{j-1}\left[
	|Y_j|^{2p-2}\langle Y_j,\Delta_j\rangle
	\right]
	\nonumber\\
	&\quad
	+
	C_p\,
	\mathbb E_{j-1}\left[
	|Y_j|^{2p-2}|\Delta_j|^2
	\right]
	\nonumber\\
	&\quad
	+
	C_p\,
	\mathbb E_{j-1}\left[
	|\Delta_j|^{2p}
	\right].
\end{align}
We estimate the three expectation terms in \eqref{eq:lem43-after-taylor}
separately. Since \(Y_j\) is \(\mathcal F_{j-1}^h\)-measurable and the Brownian
increments and iterated It\^o integrals in \(G_j\) and \(M_j\) have zero
conditional mean with respect to \(\mathcal F_{j-1}^h\), we have
$$
	\mathbb E_{j-1}[G_j]=0,
	\qquad
	\mathbb E_{j-1}[M_j]=0.
$$
Therefore, by \eqref{eq:lem43-delta-def},
\begin{equation}
	\label{eq:lem43-linear-term-start}
	\mathbb E_{j-1}\left[
	|Y_j|^{2p-2}\langle Y_j,\Delta_j\rangle
	\right]
	=
	h_j\,
	\mathbb E_{j-1}\left[
	|Y_j|^{2p-2}\langle Y_j,F_j\rangle
	\right].
\end{equation}
We decompose
\begin{equation}
	\label{eq:lem43-YF-decomp}
	\langle Y_j,F_j\rangle
	=
	\langle V_j,F_j\rangle
	-
	\langle V_j-Y_j,F_j\rangle .
\end{equation}
By the one-sided estimate for the projected drift \eqref{eq:lem43-one-sided-fh},
$$
	\langle V_j,F_j\rangle
	=
	\langle V_j,f_{h_j}(t_{j,\tau},V_j)\rangle
	\le
	C(1+|V_j|^2).
$$
Using \eqref{eq:lem43-stage-second-final}, we obtain
\begin{align}
	\label{eq:lem43-linear-main}
	h_j\,
	\mathbb E_{j-1}\left[
	|Y_j|^{2p-2}\langle V_j,F_j\rangle
	\right]
	&\le
	C h_j |Y_j|^{2p-2}
	\mathbb E_{j-1}\left[
	1+|V_j|^2
	\right]
	\nonumber\\
	&\le
	C h_j |Y_j|^{2p-2}(1+|Y_j|^2)
	\nonumber\\
	&\le
	C h_j(1+|Y_j|^{2p}).
\end{align}
For the remaining part, by \eqref{eq:lem43-fh-bound}, we have
\begin{equation}
	\label{eq:lem43-F-bound-ref}
	|F_j|\le C h_j^{-a}.
\end{equation}
Hence, by \eqref{eq:lem43-stage-increment-first},
\begin{align}
	h_j\,
	\mathbb E_{j-1}\left[
	|Y_j|^{2p-2}|\langle V_j-Y_j,F_j\rangle|
	\right]&\le
	C h_j |Y_j|^{2p-2}h_j^{-a}
	\mathbb E_{j-1}\left[
	|V_j-Y_j|
	\right]
	\nonumber\\
	&\le
	C |Y_j|^{2p-2}
	\left(
	h_j^{2-2a}
	+
	h_j^{3/2-a}(1+|Y_j|)
	\right)\nonumber.
\end{align}
Since \(a<1/2\), we have \(2-2a>1\) and \(3/2-a>1\). Therefore, for
\(h_j\le1\),
$$
	h_j^{2-2a}\le h_j,
	\qquad
	h_j^{3/2-a}\le h_j.
$$
Using also
$$
	|Y_j|^{2p-2}+|Y_j|^{2p-1}
	\le
	C(1+|Y_j|^{2p}),
$$
we get
\begin{equation}
	\label{eq:lem43-linear-rem-final}
	h_j\,
	\mathbb E_{j-1}\left[
	|Y_j|^{2p-2}|\langle V_j-Y_j,F_j\rangle|
	\right]
	\le
	C h_j(1+|Y_j|^{2p}).
\end{equation}
Combining \eqref{eq:lem43-linear-term-start},
\eqref{eq:lem43-YF-decomp}, \eqref{eq:lem43-linear-main}, and
\eqref{eq:lem43-linear-rem-final}, we obtain
\begin{equation}
	\label{eq:lem43-linear-final}
	\mathbb E_{j-1}\left[
	|Y_j|^{2p-2}\langle Y_j,\Delta_j\rangle
	\right]
	\le
	C h_j(1+|Y_j|^{2p}).
\end{equation}
Next we estimate the quadratic remainder. From \eqref{eq:lem43-delta-def},
\begin{equation}
	\label{eq:lem43-delta-square}
	|\Delta_j|^2
	\le
	C\left(
	h_j^2|F_j|^2+|G_j|^2+|M_j|^2
	\right).
\end{equation}
By \eqref{eq:lem43-F-bound-ref},
\begin{equation}
	\label{eq:lem43-drift-square}
	h_j^2|F_j|^2
	\le
	C h_j^{2-2a}
	\le
	C h_j.
\end{equation}
Moreover, since \(Y_j\) is \(\mathcal F_{j-1}^h\)-measurable and Brownian
increments after \(t_{j-1}\) are independent of \(\mathcal F_{j-1}^h\), we have
$$
	\mathbb E_{j-1}\left[
	\left|I_{t_{j-1},t_j}^{(r)}\right|^2
	\right]
	=
	h_j,
	\qquad r=1,\ldots,m.
$$
Thus, using the linear growth of \(g^r\) \eqref{eq:diffusion-growth}, one can obtain
\begin{align}
	\label{eq:lem43-G-second}
	\mathbb E_{j-1}\left[
	|G_j|^2
	\right]
	&\le
	C\sum_{r=1}^m
	|g^r(t_{j-1},Y_j)|^2
	\mathbb E_{j-1}\left[
	\left|I_{t_{j-1},t_j}^{(r)}\right|^2
	\right]
\le
	C h_j(1+|Y_j|^2).
\end{align}
Next we estimate \(M_j\). Recall the definition of $g^{r_1,r_2}(t,x)$ in \eqref{partial-priduct-gr}
\[
		g^{r_1,r_2}(t,x)
=
\partial_x g^{r_1}(t,x)\,g^{r_2}(t,x).
\]
By \eqref{eq:gx-bounded} in Assumption~\ref{ass:g} and \eqref{eq:diffusion-growth} in Remark~\ref{rem:one-sided-growth}, we have
\begin{equation}
	\label{eq:lem43-g12-linear-growth}
	|g^{r_1,r_2}(t,x)|
	\le
	C(1+|x|),
	\qquad t\in[0,T],\ x\in\mathbb R^d.
\end{equation}
For the iterated It\^o integrals, the standard conditional second-moment
estimate gives
$$
	\mathbb E_{j-1}\left[
	\left|
	I_{t_{j-1},t_j}^{(r_2,r_1)}
	\right|^2
	\right]
	\le
	C h_j^2,
	\qquad r_1,r_2=1,\ldots,m.
$$
Indeed, by the It\^o isometry and the independence of Brownian increments
after \(t_{j-1}\) from \(\mathcal F_{j-1}^h\),
\[
\mathbb E_{j-1}\left[
\left|
I_{t_{j-1},t_j}^{(r_2,r_1)}
\right|^2
\right]
\le
C
\mathbb E_{j-1}\left[
\int_{t_{j-1}}^{t_j}
\left|
I_{t_{j-1},s}^{(r_2)}
\right|^2 ds
\right]
=
C\int_{t_{j-1}}^{t_j}
(s-t_{j-1})\,ds
\le
C h_j^2.
\]
Consequently,
\begin{align}
	\label{eq:lem43-M-second}
	\mathbb E_{j-1}\left[
	|M_j|^2
	\right]
	&\le
	C
	\sum_{r_1,r_2=1}^m
	|g^{r_1,r_2}(t_{j-1},Y_j)|^2
	\mathbb E_{j-1}\left[
	\left|
	I_{t_{j-1},t_j}^{(r_2,r_1)}
	\right|^2
	\right]
	\nonumber\\
	&\le
	C h_j^2(1+|Y_j|^2).
\end{align}
Using \eqref{eq:lem43-delta-square}, \eqref{eq:lem43-drift-square},
\eqref{eq:lem43-G-second}, and \eqref{eq:lem43-M-second}, we get
\begin{align}
	\label{eq:lem43-quadratic-final}
	\mathbb E_{j-1}\left[
	|Y_j|^{2p-2}|\Delta_j|^2
	\right]
	&\le
	C h_j |Y_j|^{2p-2}
	+
	C h_j |Y_j|^{2p-2}(1+|Y_j|^2)
+
	C h_j^2 |Y_j|^{2p-2}(1+|Y_j|^2)
	\nonumber\\
	&\le
	C h_j(1+|Y_j|^{2p}).
\end{align}
Finally, we estimate the higher-order remainder. By \eqref{eq:lem43-delta-def},
\begin{equation}
	\label{eq:lem43-delta-high-start}
	|\Delta_j|^{2p}
	\le
	C_p\left(
	h_j^{2p}|F_j|^{2p}
	+
	|G_j|^{2p}
	+
	|M_j|^{2p}
	\right).
\end{equation}
Using \eqref{eq:lem43-F-bound-ref},
\begin{equation}
	\label{eq:lem43-drift-high}
	h_j^{2p}|F_j|^{2p}
	\le
	C h_j^{2p(1-a)}.
\end{equation}
Since \(p\ge1\) and \(a<1/2\), we have \(2p(1-a)>1\). Hence, for
\(h_j\le1\),
\begin{equation}
	\label{eq:lem43-drift-high-final}
	h_j^{2p(1-a)}
	\le
	C h_j.
\end{equation}
We now estimate the high-order moments of \(G_j\) and \(M_j\). Since \(Y_j\) is
\(\mathcal F_{j-1}^h\)-measurable, the coefficients
\(g^r(t_{j-1},Y_j)\) are fixed under \(\mathbb E_{j-1}[\cdot]\). Using
\[
G_j
=
\sum_{r=1}^m
g^r(t_{j-1},Y_j)I_{t_{j-1},t_j}^{(r)}
\]
and the elementary inequality
\[
\left|\sum_{r=1}^m a_r\right|^{2p}
\le
C_p\sum_{r=1}^m |a_r|^{2p},
\]
we obtain
\begin{align}
	\mathbb E_{j-1}\left[|G_j|^{2p}\right]
	&\le
	C_p
	\sum_{r=1}^m
	|g^r(t_{j-1},Y_j)|^{2p}
	\mathbb E_{j-1}\left[
	\left|I_{t_{j-1},t_j}^{(r)}\right|^{2p}
	\right]\nonumber.
\end{align}
For a Brownian increment,
$$
	\mathbb E_{j-1}\left[
	\left|I_{t_{j-1},t_j}^{(r)}\right|^{2p}
	\right]
	\le
	C_p h_j^p.
$$
Moreover, by the linear growth of \(g^r\)~\eqref{eq:diffusion-growth},
$$
	|g^r(t_{j-1},Y_j)|^{2p}
	\le
	C_p(1+|Y_j|^{2p}).
$$
Combining the last three estimates, we get
\begin{equation}
	\label{eq:lem43-G-high}
	\mathbb E_{j-1}\left[
	|G_j|^{2p}
	\right]
	\le
	C_p h_j^p(1+|Y_j|^{2p})
	\le
	C_p h_j(1+|Y_j|^{2p}).
\end{equation}
For \(M_j\), using
\[
M_j
=
\sum_{r_1,r_2=1}^m
g^{r_1,r_2}(t_{j-1},Y_j)
I_{t_{j-1},t_j}^{(r_2,r_1)}
\]
and the same elementary inequality, we obtain
\begin{align}
	\label{eq:lem43-M-high-detail}
	\mathbb E_{j-1}\left[|M_j|^{2p}\right]
	&\le
	C_p
	\sum_{r_1,r_2=1}^m
	|g^{r_1,r_2}(t_{j-1},Y_j)|^{2p}
	\mathbb E_{j-1}\left[
	\left|
	I_{t_{j-1},t_j}^{(r_2,r_1)}
	\right|^{2p}
	\right].
\end{align}
By \eqref{eq:lem43-g12-linear-growth},
\begin{equation}
	\label{eq:lem43-g12-high-linear-growth}
	|g^{r_1,r_2}(t_{j-1},Y_j)|^{2p}
	\le
	C_p(1+|Y_j|^{2p}).
\end{equation}
For the iterated It\^o integral, since
\[
I_{t_{j-1},t_j}^{(r_2,r_1)}
=
\int_{t_{j-1}}^{t_j}
I_{t_{j-1},s}^{(r_2)}\,dW^{r_1}(s)
\]
depends only on the Brownian increments over \([t_{j-1},t_j]\), it is
independent of \(\mathcal F_{j-1}^h\). Hence
\[
\mathbb E_{j-1}
\left[
\left|
I_{t_{j-1},t_j}^{(r_2,r_1)}
\right|^{2p}
\right]
=
\mathbb E
\left[
\left|
I_{t_{j-1},t_j}^{(r_2,r_1)}
\right|^{2p}
\right].
\]
By the Burkholder--Davis--Gundy inequality and Jensen's inequality,
\begin{align}
	\label{eq:lem43-iterated-high-moment}
	\mathbb E
	\left[
	\left|
	I_{t_{j-1},t_j}^{(r_2,r_1)}
	\right|^{2p}
	\right]\nonumber
	&\le
	C_p
	\mathbb E
	\left[
	\left(
	\int_{t_{j-1}}^{t_j}
	\left|
	I_{t_{j-1},s}^{(r_2)}
	\right|^2 ds
	\right)^p
	\right]
	\\\nonumber
	&\le
	C_p h_j^{p-1}
	\int_{t_{j-1}}^{t_j}
	\mathbb E
	\left[
	\left|
	I_{t_{j-1},s}^{(r_2)}
	\right|^{2p}
	\right]ds
	\\\nonumber
	&\le
	C_p h_j^{p-1}
	\int_{t_{j-1}}^{t_j}
	(s-t_{j-1})^p\,ds
	\\
	&\le
	C_p h_j^{2p}.
\end{align}
Combining \eqref{eq:lem43-M-high-detail},
\eqref{eq:lem43-g12-high-linear-growth}, and
\eqref{eq:lem43-iterated-high-moment}, we obtain
\begin{equation}
	\label{eq:lem43-M-high}
	\mathbb E_{j-1}\left[
	|M_j|^{2p}
	\right]
	\le
	C_p h_j^{2p}(1+|Y_j|^{2p})
	\le
	C_p h_j(1+|Y_j|^{2p}).
\end{equation}
Combining \eqref{eq:lem43-delta-high-start}--\eqref{eq:lem43-G-high}, and \eqref{eq:lem43-M-high}, we have
\begin{equation}
	\label{eq:lem43-high-final}
	\mathbb E_{j-1}\left[
	|\Delta_j|^{2p}
	\right]
	\le
	C_p h_j(1+|Y_j|^{2p}).
\end{equation}
Substituting \eqref{eq:lem43-linear-final}, \eqref{eq:lem43-quadratic-final},
and \eqref{eq:lem43-high-final} into \eqref{eq:lem43-after-taylor}, we conclude
that
\begin{equation}
	\label{eq:lem43-one-step-moment}
	\mathbb E_{j-1}\left[
	|X_h^j|^{2p}
	\right]
	\le
	(1+Ch_j)|Y_j|^{2p}+Ch_j.
\end{equation}
Since \(Y_j=X_h^{j-1}\), taking the full expectation gives
$$
	\mathbb E|X_h^j|^{2p}
	\le
	(1+Ch_j)\mathbb E|X_h^{j-1}|^{2p}
	+
	Ch_j.
$$
By the discrete Gronwall inequality and Assumption~\ref{ass:X0}, we obtain
$$
	\sup_{0\le j\le N}\mathbb E|X_h^j|^{2p}\le C_p.
$$
This proves \eqref{eq:lem43-moment-bound}.
	\end{proof}
%The randomized quadrature rule for stochastic processes, developed in
%\cite[Theorem 4.1]{BeynKru}, is a key tool in the analysis of randomized
%Milstein methods. The following lemma provides a local version of this
%quadrature error estimate. It is in the spirit of \cite[Theorem 4.1]{BeynKru}
%and will play an important role in our convergence analysis.
The randomized quadrature rule for stochastic processes, developed in
\cite[Theorem 4.1]{BeynKru}, is one of the main tools in the analysis of
randomized Milstein methods. In the present work, we only need a local form of
this estimate. The following lemma gives the corresponding local randomized
quadrature error bound, which will be used later in the local residual analysis.
\begin{lemma}[Exact drift temporal regularity and local randomized quadrature error]
	\label{lem:quadrature}
	Let Assumptions~\ref{ass:X0}--\ref{ass:g} hold, and let \(X\) be the exact
	solution of \eqref{eq:sde}. Define
	\[
	\mathcal Y(t):=f(t,X(t)),\qquad t\in[0,T].
	\]
	Then, for every \(q\ge2\), there exists a constant \(C_q>0\) such that
	\begin{equation}
		\label{eq:Y-holder-Lq}
		\|\mathcal Y(t)-\mathcal Y(s)\|_{L^q(\Omega_W)}
		\le
		C_q|t-s|^{1/2},
		\qquad s,t\in[0,T].
	\end{equation}
	Moreover, with \(t_{j,\tau}:=t_{j-1}+\tau_jh_j\) and
	\[
	A_j^\tau
	:=
	\int_{t_{j-1}}^{t_j}\mathcal Y(s)\,ds
	-
	h_j \mathcal Y(t_{j,\tau}),
	\]
	we have
	\begin{equation}
		\label{eq:local-quadrature-error-Lq}
		\|A_j^\tau\|_{L^q(\Omega)}
		\le
		C_qh_j^{3/2}.
	\end{equation}
	In addition,
	\begin{equation}
		\label{eq:local-quadrature-error-mean}
		\mathbb E\left[A_j^\tau\mid\mathcal F_{j-1}^h\right]=0.
	\end{equation}
\end{lemma}
\begin{proof}
	We first prove \eqref{eq:Y-holder-Lq}. For \(s,t\in[0,T]\), we write
	\[
	\mathcal Y(t)-\mathcal Y(s)
	=
	f(t,X(t))-f(t,X(s))
	+
	f(t,X(s))-f(s,X(s)).
	\]
	By the polynomial Lipschitz condition on \(f\)~\eqref{eq:drift-poly-lip}, H\"older's inequality, and
	Lemma~\ref{lem:exact-moment},
	\[
	\begin{aligned}
		&\|f(t,X(t))-f(t,X(s))\|_{L^q(\Omega_W)}
		\\
		&\le
		C
		\left\|
		\bigl(1+|X(t)|^\kappa+|X(s)|^\kappa\bigr)
		|X(t)-X(s)|
		\right\|_{L^q(\Omega_W)}
		\\
		&\le
		C_q\|X(t)-X(s)\|_{L^{2q}(\Omega_W)}
		\\
		&\le
		C_q|t-s|^{1/2}.
	\end{aligned}
	\]
	Moreover, by the time regularity assumption on \(f\)~\eqref{eq:drift-time-holder} and
	Lemma~\ref{lem:exact-moment},
	\[
	\begin{aligned}
		&\|f(t,X(s))-f(s,X(s))\|_{L^q(\Omega_W)}
		\\
		&\le
		C
		\left\|
		1+|X(s)|^{\kappa+1}
		\right\|_{L^q(\Omega_W)}
		|t-s|^{1/2}
		\\
		&\le
		C_q|t-s|^{1/2}.
	\end{aligned}
	\]
	Combining the above two estimates gives \eqref{eq:Y-holder-Lq}.
	
	Next, we estimate the local randomized quadrature error. Since
	\[
	A_j^\tau
	=
	\int_{t_{j-1}}^{t_j}
	\bigl[
	\mathcal Y(s)-\mathcal Y(t_{j,\tau})
	\bigr]\,ds,
	\]
	Minkowski's inequality gives
	\[
	\|A_j^\tau\|_{L^q(\Omega)}
	\le
	\int_{t_{j-1}}^{t_j}
	\|\mathcal Y(s)-\mathcal Y(t_{j,\tau})\|_{L^q(\Omega)}
	\,ds.
	\]
	For fixed \(s\in[t_{j-1},t_j]\), by the uniform distribution of \(\tau_j\) and
	\eqref{eq:Y-holder-Lq},
	\[
	\begin{aligned}
		&
		\|\mathcal Y(s)-\mathcal Y(t_{j,\tau})\|_{L^q(\Omega)}^q
		\\
		&=
		\int_0^1
		\mathbb E_W
		\left|
		\mathcal Y(s)-\mathcal Y(t_{j-1}+\theta h_j)
		\right|^q
		\,d\theta
		\\
		&\le
		C_q
		\int_0^1
		|s-(t_{j-1}+\theta h_j)|^{q/2}
		\,d\theta
		\\
		&\le
		C_q h_j^{q/2}.
	\end{aligned}
	\]
	Hence
	\[
	\|\mathcal Y(s)-\mathcal Y(t_{j,\tau})\|_{L^q(\Omega)}
	\le
	C_q h_j^{1/2}.
	\]
	Therefore,
	\[
	\|A_j^\tau\|_{L^q(\Omega)}
	\le
	C_q\int_{t_{j-1}}^{t_j}h_j^{1/2}\,ds
	=
	C_qh_j^{3/2}.
	\]
	This proves \eqref{eq:local-quadrature-error-Lq}.
	
	Finally, we prove the conditional mean property. Since \(\tau_j\) is independent
	of \(\mathcal F_T^W\otimes\mathcal F_{j-1}^\tau\) and is uniformly distributed
	on \((0,1)\), we have
	\[
	\begin{aligned}
		&
		\mathbb E\left[
		h_j\mathcal Y(t_{j,\tau})
		\,\middle|\,
		\mathcal F_T^W\otimes\mathcal F_{j-1}^\tau
		\right]
		\\
		&=
		h_j\int_0^1
		\mathcal Y(t_{j-1}+\theta h_j)\,d\theta
		=
		\int_{t_{j-1}}^{t_j}
		\mathcal Y(s)\,ds.
	\end{aligned}
	\]
	Thus
	\[
	\mathbb E\left[
	A_j^\tau
	\,\middle|\,
	\mathcal F_T^W\otimes\mathcal F_{j-1}^\tau
	\right]
	=0.
	\]
	Since
	\[
	\mathcal F_{j-1}^h
	=
	\mathcal F_{t_{j-1}}^W\otimes\mathcal F_{j-1}^\tau
	\subset
	\mathcal F_T^W\otimes\mathcal F_{j-1}^\tau,
	\]
	the tower property yields
	\[
	\mathbb E\left[
	A_j^\tau
	\,\middle|\,
	\mathcal F_{j-1}^h
	\right]
	=0.
	\]
	This proves \eqref{eq:local-quadrature-error-mean}.
\end{proof}
\begin{lemma}[Local Milstein diffusion residual]
	\label{lem:milstein-residual}
	Let Assumptions~\ref{ass:X0}--\ref{ass:g} hold, and let \(X\) be the exact
	solution of \eqref{eq:sde}. Define
	\begin{align}
		\label{eq:milstein-residual-def}
		\mathcal M_j
		:=
		&
		\sum_{r=1}^m
		\int_{t_{j-1}}^{t_j}
		g^r(s,X(s))\,dW^r(s)
-
		\sum_{r=1}^m
		g^r(t_{j-1},X(t_{j-1}))
		I_{t_{j-1},t_j}^{(r)}
		\nonumber\\
		&-
		\sum_{r_1,r_2=1}^m
		g^{r_1,r_2}(t_{j-1},X(t_{j-1}))
		I_{t_{j-1},t_j}^{(r_2,r_1)} .
	\end{align}
	Then there exists a constant \(C>0\), independent of \(h\) and \(j\), such that
	\begin{equation}
		\label{eq:milstein-residual-local-L2}
		\|\mathcal M_j\|_{L^2(\Omega_W)}
		\le
		Ch_j^{3/2},
	\end{equation}
	and
	\begin{equation}
		\label{eq:milstein-residual-local-mean}
		\mathbb E\left[\mathcal M_j\mid\mathcal F_{j-1}^h\right]=0.
	\end{equation}
\end{lemma}

\begin{proof}
The proof is based on the Milstein diffusion residual decomposition used in
\cite[Lemma 6.2]{KruWu2019}, but we write it directly in the present notations and only in
the \(L^2\)-setting. The projection is applied only to the drift coefficient, and therefore it does
	not enter the diffusion residual. For \(s\in[t_{j-1},t_j]\), define
	\begin{align}
		R_j^r(s)
		:=
		&
		g^r(s,X(s))
		-
		g^r(t_{j-1},X(t_{j-1}))-
		\sum_{r_2=1}^m
		g^{r,r_2}(t_{j-1},X(t_{j-1}))
		I_{t_{j-1},s}^{(r_2)} \nonumber.
	\end{align}
	Then, by the definition
	\[
	g^{r,r_2}(t,x)=\partial_x g^r(t,x)g^{r_2}(t,x),
	\]
	we have
	\begin{equation}
		\label{eq:Mj-Rj}
		\mathcal M_j
		=
		\sum_{r=1}^m
		\int_{t_{j-1}}^{t_j}
		R_j^r(s)\,dW^r(s).
	\end{equation}
To obtain\eqref{eq:milstein-residual-local-L2}, it is enough to prove
	\begin{equation}
		\label{eq:Rjr-target}
		\int_{t_{j-1}}^{t_j}
		\|R_j^r(s)\|_{L^2(\Omega_W)}^2\,ds
		\le
		C h_j^3,
		\qquad r=1,\ldots,m .
	\end{equation}
	We decompose \(R_j^r(s)\) into three parts:
	\begin{align}
		\label{eq:Rjr-decomp}
		R_j^r(s)
		&=
		\Bigl[
		g^r(s,X(s))-g^r(t_{j-1},X(s))
		\Bigr]
		\nonumber\\
		&\quad+
		\Bigl[
		g^r(t_{j-1},X(s))
		-
		g^r(t_{j-1},X(t_{j-1}))
		\nonumber\\
		&\qquad\quad
		-
		\partial_x g^r(t_{j-1},X(t_{j-1}))
		\bigl(X(s)-X(t_{j-1})\bigr)
		\Bigr]
		\nonumber\\
		&\quad+
		\Bigl[
		\partial_x g^r(t_{j-1},X(t_{j-1}))
		\bigl(X(s)-X(t_{j-1})\bigr)
		\nonumber\\
		&\qquad\quad
		-
		\sum_{r_2=1}^m
		g^{r,r_2}(t_{j-1},X(t_{j-1}))
		I_{t_{j-1},s}^{(r_2)}
		\Bigr]
		\nonumber\\
		&=:
		D_{j,1}^r(s)+D_{j,2}^r(s)+D_{j,3}^r(s).
	\end{align}
	We first estimate \(D_{j,1}^r\). By the time Lipschitz continuity of \(g^r\)~\eqref{eq:g-time-lip},
	\begin{align}
		\|D_{j,1}^r(s)\|_{L^2(\Omega_W)}
		&\le
		C
		\left\|
		(1+|X(s)|)(s-t_{j-1})
		\right\|_{L^2(\Omega_W)}
		\nonumber\\
		&\le
		C(s-t_{j-1})\nonumber,
	\end{align}
	where we used Lemma~\ref{lem:exact-moment}. Hence
	\begin{equation}
		\label{eq:D1-estimate}
		\int_{t_{j-1}}^{t_j}
		\|D_{j,1}^r(s)\|_{L^2(\Omega_W)}^2\,ds
		\le
		C h_j^3 .
	\end{equation}
	Next, set \(\Delta X_j(s):=X(s)-X(t_{j-1})\). By the integral form of the
	first-order Taylor expansion,
	\begin{align}
		D_{j,2}^r(s)=
		\int_0^1
		\Big[
		\partial_x g^r
		\bigl(t_{j-1},X(t_{j-1})+\theta\Delta X_j(s)\bigr)
		-
		\partial_x g^r\bigl(t_{j-1},X(t_{j-1})\bigr)
		\Big]
		\Delta X_j(s)\,d\theta \nonumber.
	\end{align}
	Since \(\partial_x g^r\) is Lipschitz continuous in the spatial variable,
	\[
	|D_{j,2}^r(s)|\le C|\Delta X_j(s)|^2.
	\]
	Therefore, by Lemma~\ref{lem:exact-moment},
	\begin{align}
		\|D_{j,2}^r(s)\|_{L^2(\Omega_W)}
		&\le
		C\|X(s)-X(t_{j-1})\|_{L^4(\Omega_W)}^2
		\nonumber\\
		&\le
		C(s-t_{j-1})\nonumber.
	\end{align}
	Consequently,
	\begin{equation}
		\label{eq:D2-estimate}
		\int_{t_{j-1}}^{t_j}
		\|D_{j,2}^r(s)\|_{L^2(\Omega_W)}^2\,ds
		\le
		C h_j^3 .
	\end{equation}
	It remains to estimate \(D_{j,3}^r\). Using the integral equation for the exact
	solution, we obtain
	\begin{align}
		D_{j,3}^r(s)
		&=
		\partial_x g^r(t_{j-1},X(t_{j-1}))
		\int_{t_{j-1}}^s
		f(u,X(u))\,du
		\nonumber\\
		&\quad+
		\sum_{r_2=1}^m
		\partial_x g^r(t_{j-1},X(t_{j-1}))
		\int_{t_{j-1}}^s
		\Bigl[
		g^{r_2}(u,X(u))
		-
		g^{r_2}(t_{j-1},X(t_{j-1}))
		\Bigr]\,dW^{r_2}(u)
		\nonumber\\
		&=:
		D_{j,3}^{r,f}(s)+D_{j,3}^{r,g}(s)\nonumber.
	\end{align}
An application of ~\eqref{eq:gx-bounded}, the polynomial growth of \(f\)~\eqref{eq:drift-growth} and
	Lemma~\ref{lem:exact-moment} yield
	\begin{align}
		\|D_{j,3}^{r,f}(s)\|_{L^2(\Omega_W)}
		&\le
		C
		\int_{t_{j-1}}^s
		\|f(u,X(u))\|_{L^2(\Omega_W)}\,du
		\nonumber\\
		&\le
			C\left(
		1+
		\sup_{t\in[0,T]}
		\|X(t)\|_{L^{2(\kappa+1)}(\Omega_W)}^{\kappa+1}
		\right)(s-t_{j-1})\nonumber\\
		&\leq
		C(s-t_{j-1})\nonumber.
	\end{align}
In fact, this is the only place in this residual estimate where the super-linear drift
	growth is used. Compared with the globally Lipschitz case in \cite[ Lemma 6.2]{KruWu2019}, the linear growth bound
	for \(f\) is replaced here by the polynomial growth bound together with the
	high-moment estimate of the exact solution. 
	Thus
	\begin{equation}
		\label{eq:D3f-estimate}
		\int_{t_{j-1}}^{t_j}
		\|D_{j,3}^{r,f}(s)\|_{L^2(\Omega_W)}^2\,ds
		\le
		C h_j^3 .
	\end{equation}
	For the stochastic part, by the boundedness of \(\partial_xg^r\) \eqref{eq:gx-bounded}, It\^o's
	isometry, and the inequality \(|a_1+\cdots+a_m|^2\le C\sum_{r_2=1}^m|a_{r_2}|^2\), we can obtain
	\begin{align}
		\label{eq:D3g-start}
		\|D_{j,3}^{r,g}(s)\|_{L^2(\Omega_W)}^2
		&\le
		C
		\sum_{r_2=1}^m
		\int_{t_{j-1}}^s
		\left\|
		g^{r_2}(u,X(u))-g^{r_2}(t_{j-1},X(t_{j-1}))
		\right\|_{L^2(\Omega_W)}^2\,du.
	\end{align}
	For \(u\in[t_{j-1},s]\), by the time Lipschitz continuity~\eqref{eq:g-time-lip} and the spatial
	Lipschitz continuity of \(g^{r_2}\)~\eqref{eq:g-lip},
	\begin{align*}
		&\left\|
		g^{r_2}(u,X(u))-g^{r_2}(t_{j-1},X(t_{j-1}))
		\right\|_{L^2(\Omega_W)}
		\\
		&\le
		C|u-t_{j-1}|\|1+|X(u)|\|_{L^2(\Omega_W)}
		+
		C\|X(u)-X(t_{j-1})\|_{L^2(\Omega_W)}
		\\
		&\le
		C|u-t_{j-1}|^{1/2},
	\end{align*}
	where Lemma~\ref{lem:exact-moment} and \eqref{eq:mesh-size-condition} were used
	in the last step. Substituting this estimate into \eqref{eq:D3g-start} gives
	\[
	\|D_{j,3}^{r,g}(s)\|_{L^2(\Omega_W)}^2
	\le
	C\int_{t_{j-1}}^s(u-t_{j-1})\,du
	\le
	C(s-t_{j-1})^2.
	\]
	Therefore,
$$
		\int_{t_{j-1}}^{t_j}
		\|D_{j,3}^{r,g}(s)\|_{L^2(\Omega_W)}^2\,ds
		\le
		C h_j^3 .
$$
	Together with \eqref{eq:D1-estimate}, \eqref{eq:D2-estimate}, and
	\eqref{eq:D3f-estimate}, this proves \eqref{eq:Rjr-target}. 
	By \eqref{eq:Mj-Rj}, It\^o's isometry, and \eqref{eq:Rjr-target},
	\begin{align}
		\mathbb E\left[|\mathcal M_j|^2\right]
		&\le
		C
		\sum_{r=1}^m
		\int_{t_{j-1}}^{t_j}
		\|R_j^r(s)\|_{L^2(\Omega_W)}^2\,ds
		\nonumber\\
		&\le
		C h_j^3 \nonumber.
	\end{align}
	Taking square roots gives \eqref{eq:milstein-residual-local-L2}. Moreover,
	\eqref{eq:Mj-Rj} shows that \(\mathcal M_j\) is a sum of stochastic integrals over
	\([t_{j-1},t_j]\) with adapted square-integrable integrands. Hence
	\[
	\mathbb E\left[\mathcal M_j\mid\mathcal F_{j-1}^h\right]=0,
	\]
	which proves \eqref{eq:milstein-residual-local-mean}.
\end{proof}
\begin{remark}\label{rem:difficulty-in-Lp}
	We close this section with a comment on the scope of the convergence result.
	The strong convergence estimate proved in this paper is formulated in the
	\(L^2\)-sense. Some of the local ingredients, such as the randomized quadrature
	error in Lemma~\ref{lem:quadrature} and the Milstein diffusion residual in
	Lemma~\ref{lem:milstein-residual}, can in principle be extended to \(L^p\)-spaces.
	However, an \(L^p\)-convergence analysis for the full projected
	drift-randomized Milstein method would require a substantially different
	stability argument. In particular, Lemma~\ref{lem:one-step-stability} below relies
	essentially on the mean-square structure and on the one-sided Lipschitz condition
	to control the drift contribution under super-linear growth. Extending this
	one-step stability estimate to \(L^p\) would lead to additional weighted
	cross terms and is beyond the scope of the present work.
\end{remark}

\section{Strong Convergence Analysis}\label{sec:convergence}
As mentioned in Remark~\ref{rem:difficulty-in-Lp}, if one follows the original framework of \cite{BeynKru}, the main difficulty in extending the randomized Milstein method to SDEs with super-linearly growing drift coefficients lies in the stability analysis. In this section, we establish several auxiliary lemmas to overcome this difficulty. These results form the technical core of the paper.

The first lemma establishes a one-step mean-square stability estimate, which is
a key ingredient in our convergence analysis. In contrast to the globally
Lipschitz setting considered in \cite{KruWu2019}, the super-linear growth of the
drift prevents us from applying the bistability estimate
\cite[Lemma 5.4]{KruWu2019} directly. Instead, we use a one-step stability
argument inspired by \cite[Lemma 1.1.3]{milsbook}. This approach is well suited
to the present projected scheme and will be combined later with local residual
estimates to derive the strong convergence result.
\begin{lemma}[One-step mean-square stability]
	\label{lem:one-step-stability}
	Let Assumptions~\ref{ass:X0}--\ref{ass:g} hold. Assume that
	 \(Y\), \(Z\) are \(\mathcal F_{j-1}^h\)-measurable random variables and for every \(q\ge1\), there exists a constant \(C_q>0\), independent
	of \(h_j\), such that
	\[
	\mathbb E\left[|Y|^q\right]+\mathbb E\left[|Z|^q\right]\le C_q .
	\]
	Set
	$
	t_{j,\tau}:=t_{j-1}+\tau_jh_j .
	$
	For \(y\in\mathbb R^d\), define
$$
		\psi_j^\tau(y)
		=
		y+
		\tau_jh_j f_{h_j}(t_{j-1},y)
		+
		\sum_{r=1}^m
		g^r(t_{j-1},y)
		I_{t_{j-1},t_{j,\tau}}^{(r)} .
$$
	Define also
	\begin{align}
		\label{eq:one-step-Phi-def}
		\Phi_j(y)
		=
		&
		h_j f_{h_j}
		\bigl(t_{j,\tau},\psi_j^\tau(y)\bigr)+
		\sum_{r=1}^m
		g^r(t_{j-1},y)
		I_{t_{j-1},t_j}^{(r)}+
		\sum_{r_1,r_2=1}^m
		g^{r_1,r_2}(t_{j-1},y)
		I_{t_{j-1},t_j}^{(r_2,r_1)} .
	\end{align}
	Then there exists a constant \(C>0\), independent of \(h_j\), such that
	\begin{equation}
		\label{eq:one-step-stability}
		\mathbb E\left[
		\left|
		Y-Z+\Phi_j(Y)-\Phi_j(Z)
		\right|^2\right]
		\le
		(1+Ch_j)\mathbb E\left[|Y-Z|^2\right]
		+
		Ch_j^3 .
	\end{equation}
	Moreover,
	\begin{equation}
		\label{eq:increment-difference-stability}
		\mathbb E
		\left[\left|
		\Phi_j(Y)-\Phi_j(Z)
		\right|^2\right]
		\le
		Ch_j\mathbb E\left[|Y-Z|^2\right] .
	\end{equation}
\end{lemma}

\begin{proof}
	Set
	\[
	\delta:=Y-Z,
	\qquad
	\Psi_Y:=\psi_j^\tau(Y),
	\qquad
	\Psi_Z:=\psi_j^\tau(Z).
	\]
	Moreover, define
	\[
	\Theta_j:=
	f_{h_j}(t_{j,\tau},\Psi_Y)
	-
	f_{h_j}(t_{j,\tau},\Psi_Z),
	\]
	\[
	\Xi_j:=
	\sum_{r=1}^m
	\bigl[
	g^r(t_{j-1},Y)-g^r(t_{j-1},Z)
	\bigr]
	I_{t_{j-1},t_j}^{(r)},
	\]
	and
	\[
	\Lambda_j:=
	\sum_{r_1,r_2=1}^m
	\bigl[
	g^{r_1,r_2}(t_{j-1},Y)
	-
	g^{r_1,r_2}(t_{j-1},Z)
	\bigr]
	I_{t_{j-1},t_j}^{(r_2,r_1)}.
	\]
	Then, by \eqref{eq:one-step-Phi-def},
	\[
	Y-Z+\Phi_j(Y)-\Phi_j(Z)
	=
	\delta+h_j\Theta_j+\Xi_j+\Lambda_j.
	\]
	We first collect some auxiliary estimates. Since \(T_{h_j}\) is non-expansive and
	\[
	f_{h_j}(t,x)=f(t,T_{h_j}(x)),
	\]
	the polynomial Lipschitz condition~\eqref{eq:drift-poly-lip} gives
	\begin{equation}
		\label{eq:fh-lip-h}
		|f_{h_j}(t,x)-f_{h_j}(t,y)|
		\le
		C h_j^{-\alpha \kappa}|x-y|.
	\end{equation}
	Indeed, \(|T_{h_j}(x)|,|T_{h_j}(y)|\le h_j^{-\alpha}\), and hence
	\[
	1+|T_{h_j}(x)|^\kappa+|T_{h_j}(y)|^\kappa
	\le
	C h_j^{-\alpha \kappa}.
	\]
	From the definition of \(\psi_j^\tau\), we have
	\begin{align}
		\label{Psi-decon}
		\Psi_Y-\Psi_Z
		&=
		\delta
		+
		\tau_jh_j
		\left[
		f_{h_j}(t_{j-1},Y)-f_{h_j}(t_{j-1},Z)
		\right]\nonumber\\
&\quad+
		\sum_{r=1}^m
		\left[
		g^r(t_{j-1},Y)-g^r(t_{j-1},Z)
		\right]
		I_{t_{j-1},t_{j,\tau}}^{(r)} .
	\end{align}
	By \eqref{eq:fh-lip-h},
\begin{equation}\label{eq:Psi-delta-f-estimate}
		\left\|
	\tau_jh_j
	\left[
	f_{h_j}(t_{j-1},Y)-f_{h_j}(t_{j-1},Z)
	\right]
	\right\|_{L^2(\Omega)}
	\le
	C h_j^{1-\alpha \kappa}\|\delta\|_{L^2(\Omega)}.
\end{equation}
	Moreover, by the global Lipschitz continuity of \(g^r\), we have
	\[
	\sum_{r=1}^m
	\left|
	g^r(t_{j-1},Y)-g^r(t_{j-1},Z)
	\right|^2
	\le
	C|Y-Z|^2
	=
	C|\delta|^2 .
	\]
	Set
	\[
	a_r:=
	g^r(t_{j-1},Y)-g^r(t_{j-1},Z),
	\qquad r=1,\ldots,m.
	\]
	Since \(Y\) and \(Z\) are \(\mathcal F_{j-1}^h\)-measurable, each \(a_r\) is
	\(\mathcal F_{j-1}^h\)-measurable. Put
	\[
	\mathcal H_j
	:=
	\mathcal F_{j-1}^h\vee\sigma(\tau_j).
	\]
	Conditioning on \(\mathcal H_j\), the value of \(\tau_j\) is fixed and the
	Brownian increments over
	\([t_{j-1},t_{j-1}+\tau_jh_j]\) are independent of \(\mathcal H_j\). Hence
	\[
	\mathbb E\left[
	I_{t_{j-1},t_{j,\tau}}^{(r)}
	I_{t_{j-1},t_{j,\tau}}^{(k)}
	\,\middle|\,
	\mathcal H_j
	\right]
	=
	\tau_jh_j\delta_{rk},
	\]
	where \(\delta_{rk}\) denotes the Kronecker delta. Therefore,
	\[
	\begin{aligned}
		&
		\mathbb E\left[
		\left|
		\sum_{r=1}^m
		a_r I_{t_{j-1},t_{j,\tau}}^{(r)}
		\right|^2
		\,\middle|\,
		\mathcal H_j
		\right]
		\\
		& =
		\sum_{r,k=1}^m
		\langle a_r,a_k\rangle
		\mathbb E\left[
		I_{t_{j-1},t_{j,\tau}}^{(r)}
		I_{t_{j-1},t_{j,\tau}}^{(k)}
		\,\middle|\,
		\mathcal H_j
		\right]
		\\
		& =
		\tau_jh_j
		\sum_{r=1}^m |a_r|^2
		\\
		& \le
		C h_j|\delta|^2 .
	\end{aligned}
	\]
	Taking expectations on both sides and using the tower property, we obtain
	\begin{align}
		\label{eq:Psi-gr-setimate}
		&
		\left\|
		\sum_{r=1}^m
		\left[
		g^r(t_{j-1},Y)-g^r(t_{j-1},Z)
		\right]
		I_{t_{j-1},t_{j,\tau}}^{(r)}
		\right\|_{L^2(\Omega)}^2\nonumber
		\\
		& =
		\mathbb E
		\left[
		\left|
		\sum_{r=1}^m
		a_r I_{t_{j-1},t_{j,\tau}}^{(r)}
		\right|^2
		\right]\nonumber
		\\
		& =
		\mathbb E
		\left[
		\mathbb E\left[
		\left|
		\sum_{r=1}^m
		a_r I_{t_{j-1},t_{j,\tau}}^{(r)}
		\right|^2
		\,\middle|\,
		\mathcal H_j
		\right]
		\right]\nonumber
		\\
		& \le
		C h_j\mathbb E|\delta|^2
		=
		C h_j\|\delta\|_{L^2(\Omega)}^2 .
	\end{align}
	Therefore, combining~\eqref{Psi-decon}, \eqref{eq:Psi-delta-f-estimate} and \eqref{eq:Psi-gr-setimate} gives
	\[
	\|\Psi_Y-\Psi_Z\|_{L^2(\Omega)}
	\le
	\left(
	1+Ch_j^{1-\alpha \kappa}+Ch_j^{1/2}
	\right)
	\|\delta\|_{L^2(\Omega)}.
	\]
	Since \(0<h_j\le1\) and $1-\alpha \kappa>0$, we obtain
	\begin{equation}
		\label{eq:psi-diff-estimate}
		\|\Psi_Y-\Psi_Z\|_{L^2(\Omega)}
		\le
		C\|\delta\|_{L^2(\Omega)}.
	\end{equation}
	Similarly, by \eqref{Psi-decon}, we can prove that
	\begin{align}
		\label{eq:psi-diff-perturb}
		\|\delta-(\Psi_Y-\Psi_Z)\|_{L^2(\Omega)}
		&\le
		C h_j^{1-\alpha \kappa}\|\delta\|_{L^2(\Omega)}
		+
		C h_j^{1/2}\|\delta\|_{L^2(\Omega)}
		\nonumber\\
		&\le
		C h_j^{1/2}\|\delta\|_{L^2(\Omega)},
	\end{align}
	where we used
	\[
	\alpha \kappa<\frac12,
	\]
	which follows from \(\alpha<1/(2(\kappa+1))\).
	
	Now set
	\[
	P_Y:=T_{h_j}(\Psi_Y),
	\qquad
	P_Z:=T_{h_j}(\Psi_Z).
	\]
	Then
	\[
	\Theta_j=f(t_{j,\tau},P_Y)-f(t_{j,\tau},P_Z).
	\]
	We estimate the drift cross term. Decompose
	\begin{align}
		\label{eq:drift-cross-decomposition}
		\langle \delta,\Theta_j\rangle
		&=
		\left\langle
		P_Y-P_Z,
		f(t_{j,\tau},P_Y)-f(t_{j,\tau},P_Z)
		\right\rangle
		\nonumber\\
		&\quad+
		\left\langle
		\delta-(P_Y-P_Z),
		f(t_{j,\tau},P_Y)-f(t_{j,\tau},P_Z)
		\right\rangle .
	\end{align}
	By the one-sided Lipschitz condition and \eqref{eq:Th-nonexpansive} in Lemma~\ref{lem:projection-error}, we have
	\[
	\left\langle
	P_Y-P_Z,
	f(t_{j,\tau},P_Y)-f(t_{j,\tau},P_Z)
	\right\rangle
	\le
	C|P_Y-P_Z|^2\le C |\Psi_Y-\Psi_Z|^2.
	\]
	Therefore, by \eqref{eq:psi-diff-estimate},
	\begin{equation}\label{eq:Delta-P-and-f}
		\mathbb E
\left[
\left\langle
P_Y-P_Z,
f(t_{j,\tau},P_Y)-f(t_{j,\tau},P_Z)
\right\rangle
\right]
\le
C\mathbb E|\delta|^2 .
	\end{equation}
	
	Next,
\begin{equation}\label{Delta-P-decomposition}
	\delta-(P_Y-P_Z)
=
\delta-(\Psi_Y-\Psi_Z)
+
(\Psi_Y-P_Y)
-
(\Psi_Z-P_Z).
\end{equation}
	By \eqref{eq:fh-lip-h} and \eqref{eq:psi-diff-estimate},
	\begin{equation}
		\label{eq:Thetaj-L2-bound}
		\|\Theta_j\|_{L^2(\Omega)}
		\le
		C h_j^{-\alpha \kappa}\|\delta\|_{L^2(\Omega)}.
	\end{equation}
	Hence, using H\"{o}lder's inequality \eqref{eq:psi-diff-perturb}, and \eqref{eq:Thetaj-L2-bound}, we obtain
	\begin{align}
		\label{eq:stage-cross-bound}
		&
		\mathbb E
		\left[
		\left|
		\left\langle
		\delta-(\Psi_Y-\Psi_Z),\Theta_j
		\right\rangle
		\right|
		\right]
		\nonumber\\
		&\le
		\|\delta-(\Psi_Y-\Psi_Z)\|_{L^2(\Omega)}\|\Theta_j\|_{L^2(\Omega)}
		\nonumber\\
		&\le
		C h_j^{1/2-\alpha \kappa}\|\delta\|_{L^2(\Omega)}^2
		\nonumber\\
		&\le
		C\mathbb E\left[|\delta|^2\right] .
	\end{align}
	
	It remains to control the projection remainders
	\[
	\Psi_Y-P_Y,
	\qquad
	\Psi_Z-P_Z.
	\]
	We first show that the stage variables have uniformly bounded moments of arbitrary
	finite order. Since \(|T_{h_j}(x)|\le |x|\), the growth condition on \(f\)~\eqref{eq:drift-growth} gives
	\[
	|f_{h_j}(t,x)|
	=
	|f(t,T_{h_j}(x))|
	\le
	C(1+|T_{h_j}(x)|^{\kappa+1})
	\le
	C(1+|x|^{\kappa+1}).
	\]
	By the definition of \(\Psi_Y\) and Minkowski's inequality,
	\[
	\begin{aligned}
		\|\Psi_Y\|_{L^q(\Omega)}
		&\le
		\|Y\|_{L^q(\Omega)}
		+
		h_j\|f_{h_j}(t_{j-1},Y)\|_{L^q(\Omega)}+
		\left\|
		\sum_{r=1}^m
		g^r(t_{j-1},Y)
		I_{t_{j-1},t_{j,\tau}}^{(r)}
		\right\|_{L^q(\Omega)}.
	\end{aligned}
	\]
Note that
	\[
	h_j\|f_{h_j}(t_{j-1},Y)\|_{L^q(\Omega)}
	\le
	Ch_j\left(1+\|Y\|_{L^{q(\kappa+1)}(\Omega)}^{\kappa+1}\right).
	\]
Moreover, by the linear growth of \(g^r\)~\eqref{eq:diffusion-growth}, using the following random-time
Brownian increment estimate, we have
\[
\left\|
\sum_{r=1}^m
g^r(t_{j-1},Y)I_{t_{j-1},t_{j,\tau}}^{(r)}
\right\|_{L^q(\Omega)}
\le
C_qh_j^{1/2}
\left(1+\|Y\|_{L^q(\Omega)}\right).
\]
Indeed, set
\[
\mathcal H_j:=\mathcal F_{j-1}^h\vee\sigma(\tau_j).
\]
Conditioning on \(\mathcal H_j\), the random variable \(\tau_j\) is fixed and
\[
I_{t_{j-1},t_{j,\tau}}^{(r)}
=
W^r(t_{j-1}+\tau_jh_j)-W^r(t_{j-1})
\]
is a Brownian increment over an interval of length \(\tau_jh_j\). Since
\(Y\) is \(\mathcal F_{j-1}^h\)-measurable, we get
\[
\begin{aligned}
	&
	\mathbb E\left[
	\left.
	\left|
	\sum_{r=1}^m
	g^r(t_{j-1},Y)I_{t_{j-1},t_{j,\tau}}^{(r)}
	\right|^q
	\right|
	\mathcal H_j
	\right]
	\\
	&\le
	C_q(\tau_jh_j)^{q/2}
	\left(
	\sum_{r=1}^m |g^r(t_{j-1},Y)|^2
	\right)^{q/2}
	\\
	&\le
	C_qh_j^{q/2}
	\left(
	1+|Y|^q
	\right).
\end{aligned}
\]
Taking expectations gives the desired estimate.
	Therefore, by the assumption $\mathbb E[|Y|^q]\leq C_q\text{, for all } q\geq1$, we have
	\[
	\|\Psi_Y\|_{L^q(\Omega)}
	\le
	\|Y\|_{L^q(\Omega)}
	+
	Ch_j\left(1+\|Y\|_{L^{q(\kappa+1)}(\Omega)}^{\kappa+1}\right)
	+
	Ch_j^{1/2}(1+\|Y\|_{L^q(\Omega)})
	\le C_q.
	\]
	The estimate for \(\Psi_Z\) is identical. Therefore, by \eqref{eq:projection-state-error-new} in
	Lemma~\ref{lem:projection-error}, choosing \(\ell\) sufficiently large,
	\begin{equation}
		\label{eq:projection-stage-tail}
		\|\Psi_Y-P_Y\|_{L^2(\Omega)}
		+
		\|\Psi_Z-P_Z\|_{L^2(\Omega)}
		\le
		C h_j^{\frac{\alpha}{2}(\ell-2)}.
	\end{equation}
	On the other hand, since \(|P_Y|,|P_Z|\le h_j^{-\alpha}\), the growth bound of
	\(f\)~\eqref{eq:drift-growth} gives
	\begin{equation}
		\label{eq:Thetaj-crude-bound}
		|\Theta_j|\leq C|1+h^{-\alpha(\kappa+1)}|\le C h_j^{-\alpha(\kappa+1)},
	\end{equation}
	where the last inequality we used the fact $h^{-\alpha(\kappa+1)}\geq1$. Choose \(\ell\) sufficiently large such that
	\[
	\frac{\alpha}{2}(\ell-2)-\alpha(\kappa+1)\ge 2.
	\]
	Then, by \eqref{eq:projection-stage-tail} and \eqref{eq:Thetaj-crude-bound},
	\begin{align}
		\label{eq:projection-cross-bound}
		&
		\mathbb E
		\left[
		\left|
		\left\langle
		(\Psi_Y-P_Y)-(\Psi_Z-P_Z),
		\Theta_j
		\right\rangle
		\right|
		\right]
		\nonumber\\
		&\le
		C h_j^{-\alpha(\kappa+1)}
		\left(
		\|\Psi_Y-P_Y\|_{L^2(\Omega)}
		+
		\|\Psi_Z-P_Z\|_{L^2(\Omega)}
		\right)
		\nonumber\\
		&\le
		C h_j^2.
	\end{align}
	Combining \eqref{eq:drift-cross-decomposition}--\eqref{Delta-P-decomposition}, \eqref{eq:stage-cross-bound} and \eqref{eq:projection-cross-bound},
	we obtain
	\begin{equation}
		\label{eq:drift-cross-final}
		\mathbb E\left[\langle\delta,\Theta_j\rangle\right]
		\le
		C\mathbb E\left[|\delta|^2\right]
		+
		C h_j^2.
	\end{equation}
	
	We also need a square estimate for the drift difference. By \eqref{eq:fh-lip-h}
	and \eqref{eq:psi-diff-estimate},
	\begin{align}
		\label{eq:drift-square-estimate}
		h_j^2\mathbb E\left[|\Theta_j|^2\right]
		&\le
		C h_j^2 h_j^{-2\alpha \kappa}
		\mathbb E\left[|\Psi_Y-\Psi_Z|^2\right]
		\nonumber\\
		&\le
		C h_j^{2-2\alpha \kappa}
		\mathbb E\left[|\delta|^2\right]
		\nonumber\\
		&\le
		C h_j\mathbb E\left[|\delta|^2\right],
	\end{align}
	where we used \(\alpha \kappa<1/2\).
	
	For the diffusion part, recall our notations:
		\[
	\Xi_j:=
	\sum_{r=1}^m
	\bigl[
	g^r(t_{j-1},Y)-g^r(t_{j-1},Z)
	\bigr]
	I_{t_{j-1},t_j}^{(r)},
	\]
	\[
	\Lambda_j:=
	\sum_{r_1,r_2=1}^m
	\bigl[
	g^{r_1,r_2}(t_{j-1},Y)
	-
	g^{r_1,r_2}(t_{j-1},Z)
	\bigr]
	I_{t_{j-1},t_j}^{(r_2,r_1)}.
	\]
	The global Lipschitz continuity of \(g^r\)~\eqref{eq:g-lip} gives
	\begin{equation}
		\label{eq:Xi-square-estimate}
		\mathbb E\left[|\Xi_j|^2\right]
		\le
		C h_j\mathbb E\left[|\delta|^2\right].
	\end{equation}
	Similarly, by the global Lipschitz continuity of \(g^{r_1,r_2}\)~\eqref{eq:g12-lip} and the
	second-moment estimate of the iterated It\^o integrals,
	\begin{equation}
		\label{eq:Lambda-square-estimate}
		\mathbb E\left[|\Lambda_j|^2\right]
		\le
		C h_j^2\mathbb E\left[|\delta|^2\right]
		\le
		C h_j\mathbb E\left[|\delta|^2\right].
	\end{equation}
	
	Finally, expanding the square and using
	\[
	|h_j\Theta_j+\Xi_j+\Lambda_j|^2
	\le
	C h_j^2|\Theta_j|^2+C|\Xi_j|^2+C|\Lambda_j|^2,
	\]
	we get
	\begin{align}	
		\mathbb E
		\left[\left|
		\delta+h_j\Theta_j+\Xi_j+\Lambda_j
		\right|^2\right]
		&\le
		\mathbb E\left[|\delta|^2\right]
		+
		2h_j\mathbb E\left[\langle\delta,\Theta_j\rangle\right]
		+
		2\mathbb E\left[\langle\delta,\Xi_j\rangle\right]
		+
		2\mathbb E\left[\langle\delta,\Lambda_j\rangle\right]
		\nonumber\\
		&\quad+
		C h_j^2\mathbb E\left[|\Theta_j|^2\right]
		+
		C\mathbb E\left[|\Xi_j|^2\right]
		+
		C\mathbb E\left[|\Lambda_j|^2\right] \nonumber.
	\end{align}
	Since \(\delta\) is \(\mathcal F_{j-1}^h\)-measurable and the Brownian increments
	and iterated It\^o integrals have zero conditional mean, we have
	\[
	\mathbb E\left[\langle\delta,\Xi_j\rangle\right]=0,
	\qquad
	\mathbb E\left[\langle\delta,\Lambda_j\rangle\right]=0.
	\]
	Using \eqref{eq:drift-cross-final}, \eqref{eq:drift-square-estimate},
	\eqref{eq:Xi-square-estimate}, and \eqref{eq:Lambda-square-estimate}, we get
	\[
	\mathbb E
	\left[\left|
	\delta+h_j\Theta_j+\Xi_j+\Lambda_j
	\right|^2\right]
	\le
	(1+Ch_j)\mathbb E\left[|\delta|^2\right]
	+
	Ch_j^3.
	\]
	This proves \eqref{eq:one-step-stability}.
	
	It remains to prove \eqref{eq:increment-difference-stability}. Since
	\[
	\Phi_j(Y)-\Phi_j(Z)=h_j\Theta_j+\Xi_j+\Lambda_j,
	\]
	we have
	\[
	\mathbb E\left[|\Phi_j(Y)-\Phi_j(Z)|^2\right]
	\le
	C h_j^2\mathbb E\left[|\Theta_j|^2\right]
	+
	C\mathbb E\left[|\Xi_j|^2\right]
	+
	C\mathbb E\left[|\Lambda_j|^2\right] .
	\]
	Using \eqref{eq:drift-square-estimate}, \eqref{eq:Xi-square-estimate}, and
	\eqref{eq:Lambda-square-estimate}, we obtain
	\[
	\mathbb E\left[|\Phi_j(Y)-\Phi_j(Z)|^2\right]
	\le
	Ch_j\mathbb E\left[|Y-Z|^2\right] .
	\]
	This proves \eqref{eq:increment-difference-stability}.
\end{proof}
The following lemma provides a local estimate for the internal randomized stage.
It shows that, when the stage is initialized at the exact solution
\(X(t_{j-1})\), it approximates the exact solution at the randomized time
\(t_{j,\tau}\) with order one in \(L^q(\Omega)\). This estimate will be used as a
local consistency ingredient in the convergence analysis.

\begin{lemma}[Local stage estimate]
	\label{lem:local-stage-estimate}
	Let Assumptions~\ref{ass:X0}--\ref{ass:g} hold. Let
$
	t_{j,\tau}:=t_{j-1}+\tau_jh_j .
$
	We use the same internal stage map \(\psi_j^\tau(\cdot)\) as in
	Lemma~\ref{lem:one-step-stability}. In particular,
	\[
	\begin{aligned}
		\psi_j^\tau\bigl(X(t_{j-1})\bigr)
		&=
		X(t_{j-1})
		+
		\tau_jh_j f_{h_j}(t_{j-1},X(t_{j-1}))
		\\
		&\quad+
		\sum_{r=1}^m
		g^r(t_{j-1},X(t_{j-1}))
		I_{t_{j-1},t_{j,\tau}}^{(r)} .
	\end{aligned}
	\]
	Then, for every \(q\ge2\), there exists a constant \(C_q>0\), independent of
	\(h_j\), such that
	\begin{equation}
		\label{eq:local-stage-moment}
		\left\|
		\psi_j^\tau\bigl(X(t_{j-1})\bigr)
		\right\|_{L^q(\Omega)}
		\le C_q,
	\end{equation}
	and
	\begin{equation}
		\label{eq:local-stage-error}
		\left\|
		X(t_{j,\tau})
		-
		\psi_j^\tau\bigl(X(t_{j-1})\bigr)
		\right\|_{L^q(\Omega)}
		\le C_q h_j .
	\end{equation}
\end{lemma}
\begin{proof}
	For the moment estimate, by the definition of
	\(\psi_j^\tau\bigl(X(t_{j-1})\bigr)\) and Minkowski's inequality,
	\[
	\begin{aligned}
		\left\|
		\psi_j^\tau\bigl(X(t_{j-1})\bigr)
		\right\|_{L^q(\Omega)}
		&\le
		\|X(t_{j-1})\|_{L^q(\Omega)}
		+
		h_j
		\|f_{h_j}(t_{j-1},X(t_{j-1}))\|_{L^q(\Omega)}
		\\
		&\quad+
		\left\|
		\sum_{r=1}^m
		g^r(t_{j-1},X(t_{j-1}))
		I_{t_{j-1},t_{j,\tau}}^{(r)}
		\right\|_{L^q(\Omega)} .
	\end{aligned}
	\]
	For the last term, using the same conditioning argument as in the proof of
	Lemma~\ref{lem:one-step-stability}, with
$
	\mathcal H_j:=\mathcal F_{j-1}^h\vee\sigma(\tau_j),
$
	we have
	\[
	\begin{aligned}
		&\mathbb E\left[
		\left.
		\left|
		\sum_{r=1}^m
		g^r(t_{j-1},X(t_{j-1}))
		I_{t_{j-1},t_{j,\tau}}^{(r)}
		\right|^q
		\right|
		\mathcal H_j
		\right]\le
		C_q h_j^{q/2}
		\left(
		\sum_{r=1}^m
		|g^r(t_{j-1},X(t_{j-1}))|^2
		\right)^{q/2}.
	\end{aligned}
	\]
	Taking expectations and using the linear growth of \(g^r\) \eqref{eq:diffusion-growth} and Lemma~\ref{lem:exact-moment}, we obtain
	\[
	\left\|
	\sum_{r=1}^m
	g^r(t_{j-1},X(t_{j-1}))
	I_{t_{j-1},t_{j,\tau}}^{(r)}
	\right\|_{L^q(\Omega)}
	\le
	C_q h_j^{1/2}
	\left(
	1+\|X(t_{j-1})\|_{L^q(\Omega)}
	\right).
	\]
	Therefore, by the fact $T_{h_j}(x)\leq x$, \eqref{eq:drift-growth} and Lemma~\ref{lem:exact-moment},
\begin{align}
		\|f_{h_j}(t_{j-1},X(t_{j-1}))\|_{L^q(\Omega)}&\leq \|f(t,T_{h_j}(X(t_{j-1})))\|_{L^q(\Omega)}\nonumber\\
	&\le
	C\left(
	1+\|X(t_{j-1})\|_{L^{q(\kappa+1)}(\Omega)}^{\kappa+1}
	\right)\nonumber\\
	&\le C_q\nonumber.
\end{align}
	Combining the above estimates and using Lemma~\ref{lem:exact-moment} again, we
	obtain
	\[
	\left\|
	\psi_j^\tau\bigl(X(t_{j-1})\bigr)
	\right\|_{L^q(\Omega)}
	\le C_q,
	\]
	which proves \eqref{eq:local-stage-moment}.
	
	We now prove \eqref{eq:local-stage-error}. By the integral form of the exact
	solution and the definition of \(\psi_j^\tau\), we have
	\[
	X(t_{j,\tau})
	-
	\psi_j^\tau\bigl(X(t_{j-1})\bigr)
	=
	\mathcal R_j^{\mathrm f}
	+
	\mathcal R_j^{\mathrm g},
	\]
	where
	\[
	\mathcal R_j^{\mathrm f}
	:=
	\int_{t_{j-1}}^{t_{j,\tau}}
	\left[
	f(s,X(s))
	-
	f_{h_j}(t_{j-1},X(t_{j-1}))
	\right]ds,
	\]
	and
	\[
	\mathcal R_j^{\mathrm g}
	:=
	\sum_{r=1}^m
	\int_{t_{j-1}}^{t_{j,\tau}}
	\left[
	g^r(s,X(s))
	-
	g^r(t_{j-1},X(t_{j-1}))
	\right]dW^r(s).
	\]
	
	For the drift part, we write
	\[
	\begin{aligned}
		\mathcal R_j^{\mathrm f}
		&=
		\int_{t_{j-1}}^{t_{j,\tau}}
		\left[
		f(s,X(s))
		-
		f(t_{j-1},X(t_{j-1}))
		\right]ds
		\\
		&\quad+
		(t_{j,\tau}-t_{j-1})
		\left[
		f(t_{j-1},X(t_{j-1}))
		-
		f_{h_j}(t_{j-1},X(t_{j-1}))
		\right].
	\end{aligned}
	\]
	Since \(t_{j,\tau}\le t_j\), by Jensen's inequality,
	\[
	\begin{aligned}
		&
		\left\|
		\int_{t_{j-1}}^{t_{j,\tau}}
		\left[
		f(s,X(s))
		-
		f(t_{j-1},X(t_{j-1}))
		\right]ds
		\right\|_{L^q(\Omega)}\le
		\int_{t_{j-1}}^{t_j}
		\left\|
		f(s,X(s))
		-
		f(t_{j-1},X(t_{j-1}))
		\right\|_{L^q(\Omega)}
		ds.
	\end{aligned}
	\]
	By \eqref{eq:Y-holder-Lq} in Lemma~\ref{lem:quadrature}, for every
	\(q\ge2\), we can derive that
	\[
	\left\|
	f(s,X(s))
	-
	f(t_{j-1},X(t_{j-1}))
	\right\|_{L^q(\Omega)}
	\le
	C_q |s-t_{j-1}|^{1/2}.
	\]
	Therefore,
	\[
	\left\|
	\int_{t_{j-1}}^{t_{j,\tau}}
	\left[
	f(s,X(s))
	-
	f(t_{j-1},X(t_{j-1}))
	\right]ds
	\right\|_{L^q(\Omega)}
	\le
	C_q h_j^{3/2}.
	\]
	Moreover, by the growth condition on \(f\)~\eqref{eq:drift-growth}, the definition of \(f_{h_j}\), and
	\(|T_{h_j}(x)|\le |x|\), we have
	\[
	\begin{aligned}
		&
		|f(t_{j-1},X(t_{j-1}))
		-
		f_{h_j}(t_{j-1},X(t_{j-1}))|
		\\
		&\le
		|f(t_{j-1},X(t_{j-1}))|
		+
		|f(t_{j-1},T_{h_j}(X(t_{j-1})))|
		\\
		&\le
		C\left(1+|X(t_{j-1})|^{\kappa+1}\right).
	\end{aligned}
	\]
	Hence, by Lemma~\ref{lem:exact-moment},
	\[
	\begin{aligned}
		&
		\|f(t_{j-1},X(t_{j-1}))
		-
		f_{h_j}(t_{j-1},X(t_{j-1}))\|_{L^q(\Omega)}
		\\
		&\le
		C\left(
		1+\|X(t_{j-1})\|_{L^{q(\kappa+1)}(\Omega)}^{\kappa+1}
		\right)
		\le C_q .
	\end{aligned}
	\]
	Since \(0\le t_{j,\tau}-t_{j-1}\le h_j\), we obtain
	\[
	\|\mathcal R_j^{\mathrm f}\|_{L^q(\Omega)}
	\le
	C_q h_j^{3/2}
	+
	C_q h_j
	\le
	C_q h_j.
	\]
	
For the diffusion part, we write
\[
\mathcal R_j^{\mathrm g}
=
\sum_{r=1}^m
\int_{t_{j-1}}^{t_j}
\mathbf 1_{\{s\le t_{j,\tau}\}}
\left[
g^r(s,X(s))-g^r(t_{j-1},X(t_{j-1}))
\right]dW^r(s).
\]
Since \(\tau_j\) is independent of the Brownian motion, the Brownian motion
remains Brownian after enlarging the filtration by \(\sigma(\tau_j)\). Hence
the Burkholder--Davis--Gundy inequality can be applied to the above stochastic
integral. Therefore,
\[
\begin{aligned}
	\|\mathcal R_j^{\mathrm g}\|_{L^q(\Omega)}
	&\le
	C_q
	\left\|
	\left(
	\int_{t_{j-1}}^{t_j}
	\sum_{r=1}^m
	\mathbf 1_{\{s\le t_{j,\tau}\}}
	\left|
	g^r(s,X(s))-g^r(t_{j-1},X(t_{j-1}))
	\right|^2 ds
	\right)^{1/2}
	\right\|_{L^q(\Omega)}
	\\
	&\le
	C_q
	\left(
	\int_{t_{j-1}}^{t_j}
	\sum_{r=1}^m
	\left\|
	g^r(s,X(s))-g^r(t_{j-1},X(t_{j-1}))
	\right\|_{L^q(\Omega)}^2 ds
	\right)^{1/2}.
\end{aligned}
\]
Here, in the last step, we used Minkowski's integral inequality and
\(\mathbf 1_{\{s\le t_{j,\tau}\}}\le1\). On the other hand,
	using the time Lipschitz continuity~\eqref{eq:g-time-lip} and the spatial Lipschitz continuity of
	\(g^r\)~\eqref{eq:g-lip}, together with Lemma~\ref{lem:exact-moment}, we get
	\[
	\begin{aligned}
		&
		\|g^r(s,X(s))-g^r(t_{j-1},X(t_{j-1}))\|_{L^q(\Omega)}
		\\
		&\le
		\|g^r(s,X(s))-g^r(t_{j-1},X(s))\|_{L^q(\Omega)}
		\\
		&\quad+
		\|g^r(t_{j-1},X(s))-g^r(t_{j-1},X(t_{j-1}))\|_{L^q(\Omega)}
		\\
		&\le
		C_q|s-t_{j-1}|
		+
		C\|X(s)-X(t_{j-1})\|_{L^q(\Omega)}
		\\
		&\le
		C_q|s-t_{j-1}|^{1/2}.
	\end{aligned}
	\]
	Therefore,
	\[
	\|\mathcal R_j^{\mathrm g}\|_{L^q(\Omega)}
	\le
	C_q
	\left(
	\int_{t_{j-1}}^{t_j}
	(s-t_{j-1})\,ds
	\right)^{1/2}
	\le
	C_q h_j.
	\]
	Combining the estimates for \(\mathcal R_j^{\mathrm f}\) and
	\(\mathcal R_j^{\mathrm g}\), we obtain
	\[
	\left\|
	X(t_{j,\tau})
	-
	\psi_j^\tau\bigl(X(t_{j-1})\bigr)
	\right\|_{L^q(\Omega)}
	\le
	C_q h_j.
	\]
	This proves \eqref{eq:local-stage-error}.
\end{proof}
%The following lemma estimates the one-step local residual obtained by applying
%the numerical one-step map to the exact solution at the previous grid point. It
%measures the discrepancy between the exact increment over one time step and the
%corresponding projected drift-randomized Milstein increment. In particular, the
%lemma provides both an \(L^2\)-estimate for the local residual and a sharper
%estimate for its conditional mean, which are the two local consistency bounds
%needed for the subsequent mean-square convergence analysis.
The following lemma provides the local residual estimates for the projected
drift-randomized Milstein method. It compares the exact one-step increment with
the numerical one-step increment initialized at the exact value
\(X(t_{j-1})\). The resulting \(L^2\)-bound and conditional mean estimate form
the local consistency part of the convergence analysis.
\begin{lemma}[Local residual estimates]
	\label{lem:local-residual}
	Let Assumptions~\ref{ass:X0}--\ref{ass:g} hold.
	For \(j=1,\ldots,N\), define
	\begin{equation}
		\label{eq:local-residual-def}
		\rho_j
		:=
		X(t_j)-X(t_{j-1})-\Phi_j(X(t_{j-1})),
	\end{equation}
	where \(\Phi_j(\cdot)\) is defined in \eqref{eq:one-step-Phi-def} of Lemma~\ref{lem:one-step-stability}. Then
	there exists a constant \(C>0\), independent of \(h_j\), such that
	\begin{equation}
		\label{eq:local-residual-L2}
		\|\rho_j\|_{L^2(\Omega)}
		\le
		Ch_j^{3/2},
	\end{equation}
	and
	\begin{equation}
		\label{eq:local-residual-mean}
		\left\|
		\mathbb E\left[\rho_j\mid\mathcal F_{j-1}^h\right]
		\right\|_{L^2(\Omega)}
		\le
		Ch_j^2 .
	\end{equation}
\end{lemma}

\begin{proof}
	Recall that
	\[
	t_{j,\tau}:=t_{j-1}+\tau_jh_j .
	\]
	By the definition of \(\Phi_j(\cdot)\)~\eqref{eq:one-step-Phi-def}, we decompose \(\rho_j\) as
	\[
	\rho_j
	=
	\mathcal Q_j+\mathcal S_j+\mathcal P_j+\mathcal M_j,
	\]
	where
	\[
	\mathcal Q_j
	:=
	\int_{t_{j-1}}^{t_j} f(s,X(s))\,ds
	-
	h_j f(t_{j,\tau},X(t_{j,\tau})),
	\]
	\[
	\mathcal S_j
	:=
	h_j
	\left[
	f(t_{j,\tau},X(t_{j,\tau}))
	-
	f\left(t_{j,\tau},\psi_j^\tau(X(t_{j-1}))\right)
	\right],
	\]
	\[
	\mathcal P_j
	:=
	h_j
	\left[
	f\left(t_{j,\tau},\psi_j^\tau(X(t_{j-1}))\right)
	-
	f_{h_j}\left(t_{j,\tau},\psi_j^\tau(X(t_{j-1}))\right)
	\right],
	\]
	and
	\begin{align*}
		\mathcal M_j
		:=
		&
		\sum_{r=1}^m
		\int_{t_{j-1}}^{t_j}g^r(s,X(s))\,dW^r(s)
		-
		\sum_{r=1}^m
		g^r(t_{j-1},X(t_{j-1}))
		I_{t_{j-1},t_j}^{(r)}
		\\
		&-
		\sum_{r_1,r_2=1}^m
		g^{r_1,r_2}(t_{j-1},X(t_{j-1}))
		I_{t_{j-1},t_j}^{(r_2,r_1)},
	\end{align*}
is the same as in Lemma~\ref{lem:milstein-residual}.	We estimate these four terms separately.
	
	First, by Lemma~\ref{lem:quadrature}, with
	\[
	\mathcal Y(t):=f(t,X(t)),
	\]
	we have
	\begin{equation}
		\label{eq:Qj-L2-bound}
		\|\mathcal Q_j\|_{L^2(\Omega)}
		\le
		Ch_j^{3/2},
	\end{equation}
	and
	\begin{equation}
		\label{eq:Qj-conditional-zero}
		\mathbb E[\mathcal Q_j\mid\mathcal F_{j-1}^h]=0.
	\end{equation}
	
	Second, we estimate the stage residual \(\mathcal S_j\). By the polynomial
	Lipschitz condition on \(f\)~\eqref{eq:drift-poly-lip} H\"{o}lder's inequality, we can get
	\begin{align*}
		\|\mathcal S_j\|_{L^2(\Omega)}
		&\le
		Ch_j
		\left\|
		\left(
		1+|X(t_{j,\tau})|^\kappa
		+
		\left|\psi_j^\tau(X(t_{j-1}))\right|^\kappa
		\right)
		\right.
\left.
		\left|
		X(t_{j,\tau})
		-
		\psi_j^\tau(X(t_{j-1}))
		\right|
		\right\|_{L^2(\Omega)}
		\\
		&\le
		Ch_j
		\left\|
		1+|X(t_{j,\tau})|^\kappa
		+
		\left|\psi_j^\tau(X(t_{j-1}))\right|^\kappa
		\right\|_{L^4(\Omega)}
		\left\|
		X(t_{j,\tau})
		-
		\psi_j^\tau(X(t_{j-1}))
		\right\|_{L^4(\Omega)} .
	\end{align*}
	By Lemma~\ref{lem:exact-moment} and Lemma~\ref{lem:local-stage-estimate}, the
	first factor is bounded uniformly in \(j\) and \(h_j\). Moreover, by
	Lemma~\ref{lem:local-stage-estimate},
	\[
	\left\|
	X(t_{j,\tau})
	-
	\psi_j^\tau(X(t_{j-1}))
	\right\|_{L^4(\Omega)}
	\le
	Ch_j .
	\]
	Therefore,
	\begin{equation}
		\label{eq:Sj-L2-bound}
		\|\mathcal S_j\|_{L^2(\Omega)}
		\le
		Ch_j^2 .
	\end{equation}
	
	Third, we estimate the projection residual \(\mathcal P_j\). By
	Lemma~\ref{lem:local-stage-estimate}, the random variable
$
	\psi_j^\tau(X(t_{j-1}))
$
	has moments of arbitrary finite order. Hence we may choose
	\[
	\ell\ge 2(\kappa+1)+\frac{2}{\alpha}
	\]
	such that
	\[
	\sup_{1\le j\le N}
	\left\|
	\psi_j^\tau(X(t_{j-1}))
	\right\|_{L^\ell(\Omega)}
	<\infty .
	\]
	Since the estimate \eqref{eq:projection-drift-error-order-one-new} in Lemma~\ref{lem:projection-error} is uniform in the time
	variable, we may apply it with
$
	Z=\psi_j^\tau(X(t_{j-1}))
$
	and \(t=t_{j,\tau}\). This gives
	\[
	\left\|
	f\left(t_{j,\tau},\psi_j^\tau(X(t_{j-1}))\right)
	-
	f_{h_j}\left(t_{j,\tau},\psi_j^\tau(X(t_{j-1}))\right)
	\right\|_{L^2(\Omega)}
	\le
	Ch_j .
	\]
	Consequently,
	\begin{equation}
		\label{eq:Pj-L2-bound}
		\|\mathcal P_j\|_{L^2(\Omega)}
		\le
		Ch_j^2 .
	\end{equation}
	
	Finally, by Lemma~\ref{lem:milstein-residual},
	\begin{equation}
		\label{eq:Mj-L2-bound}
		\|\mathcal M_j\|_{L^2(\Omega)}
		\le
		Ch_j^{3/2},
	\end{equation}
	and
	\begin{equation}
		\label{eq:Mj-conditional-zero}
		\mathbb E[\mathcal M_j\mid\mathcal F_{j-1}^h]=0 .
	\end{equation}
	
	Combining \eqref{eq:Qj-L2-bound}, \eqref{eq:Sj-L2-bound},
	\eqref{eq:Pj-L2-bound}, and \eqref{eq:Mj-L2-bound}, we obtain
	\[
	\|\rho_j\|_{L^2(\Omega)}
	\le
	Ch_j^{3/2}.
	\]
	
	It remains to estimate the conditional mean. By
	\eqref{eq:Qj-conditional-zero} and \eqref{eq:Mj-conditional-zero},
	\[
	\mathbb E[\rho_j\mid\mathcal F_{j-1}^h]
	=
	\mathbb E[\mathcal S_j+\mathcal P_j\mid\mathcal F_{j-1}^h].
	\]
	Therefore, by Jensen's inequality for conditional expectations,
	\[
	\begin{aligned}
		\left\|
		\mathbb E[\rho_j\mid\mathcal F_{j-1}^h]
		\right\|_{L^2(\Omega)}
		&\le
		\left\|
		\mathbb E[\mathcal S_j\mid\mathcal F_{j-1}^h]
		\right\|_{L^2(\Omega)}
		+
		\left\|
		\mathbb E[\mathcal P_j\mid\mathcal F_{j-1}^h]
		\right\|_{L^2(\Omega)}
		\\
		&\le
		\|\mathcal S_j\|_{L^2(\Omega)}
		+
		\|\mathcal P_j\|_{L^2(\Omega)}
		\\
		&\le
		Ch_j^2 .
	\end{aligned}
	\]
	This proves the lemma.
\end{proof}

We are now ready to prove the main convergence result. The proof combines the
one-step mean-square stability estimate with the local residual bounds derived
above. More precisely, the local residual estimate provides the consistency
error of one step initialized at the exact solution, while the one-step
stability estimate allows this local error to be propagated along the temporal
grid. A discrete Gronwall argument then yields the global strong convergence
order.
\begin{theorem}[Strong convergence of order one in \(L^2\)]
	\label{thm:main-convergence}
	Let Assumptions~\ref{ass:X0}--\ref{ass:g} hold
	and \(X\) be the exact solution of \eqref{eq:sde}. Let
	\((X_h^j)_{j=0}^N\) be generated by the projected drift-randomized Milstein
	method \eqref{eq:Xtau}-\eqref{eq:scheme}. Then there exists a constant \(C>0\), independent of \(h\), such that
	\begin{equation}
		\label{eq:main-convergence-final}
		\max_{0\le j\le N}
		\|X(t_j)-X_h^j\|_{L^2(\Omega)}
		\le
		Ch.
	\end{equation}
\end{theorem}

\begin{proof}
	Let
	\[
	e_j:=X(t_j)-X_h^j,
	\qquad j=0,1,\ldots,N.
	\]
	By Lemma~\ref{lem:local-residual}, the exact solution satisfies
	\[
	X(t_j)
	=
	X(t_{j-1})
	+
	\Phi_j(X(t_{j-1}))
	+
	\rho_j ,
	\]
where $\rho_j$ is defined by \eqref{eq:local-residual-def}.	On the other hand, recall the definition of \(\Phi_j(\cdot)\) in~\eqref{eq:one-step-Phi-def}, the numerical scheme can be written as
	\[
	X_h^j
	=
	X_h^{j-1}
	+
	\Phi_j(X_h^{j-1}).
	\]
	Hence
	\[
	e_j
	=
	e_{j-1}
	+
	\Phi_j(X(t_{j-1}))
	-
	\Phi_j(X_h^{j-1})
	+
	\rho_j .
	\]
	Set
	\[
	\eta_j
	:=
	e_{j-1}
	+
	\Phi_j(X(t_{j-1}))
	-
	\Phi_j(X_h^{j-1}).
	\]
	Then
	\[
	e_j=\eta_j+\rho_j.
	\]
	The moment bounds from Lemmas~\ref{lem:exact-moment} and~\ref{lem:numerical-moment} guarantees that we can apply  Lemma~\ref{lem:one-step-stability} with
	\(Y=X(t_{j-1})\) and \(Z=X_h^{j-1}\). Hence, from \eqref{eq:one-step-stability} in Lemma~\ref{lem:one-step-stability} we obtain
	\begin{equation}
		\label{eq:conv-eta-j}
		\mathbb E\left[|\eta_j|^2\right]
		\le
		(1+Ch_j)\mathbb E\left[|e_{j-1}|^2\right]
		+
		Ch_j^3 .
	\end{equation}
	Moreover, by \eqref{eq:increment-difference-stability} in Lemma~\ref{lem:one-step-stability}, we can derive
	\begin{equation}
		\label{eq:conv-Phi-diff}
		\mathbb E
		\left[\left|
		\Phi_j(X(t_{j-1}))
		-
		\Phi_j(X_h^{j-1})
		\right|^2\right]
		\le
		Ch_j\mathbb E\left[|e_{j-1}|^2\right] .
	\end{equation}
	Expanding \(e_j=\eta_j+\rho_j\), we get
\begin{equation}\label{eq:expansion-ej}
	\mathbb E\left[|e_j|^2\right]
=
\mathbb E\left[|\eta_j|^2\right]
+
2\mathbb E[\langle \eta_j,\rho_j\rangle]
+
\mathbb E\left[|\rho_j|^2\right] .
\end{equation}
	We estimate the cross term. By the definition of $\eta_j$,	we have
\begin{equation}\label{eq:product-eta-rho}
	\mathbb E[\langle \eta_j,\rho_j\rangle]
	=
	\mathbb E[\langle e_{j-1},\rho_j\rangle]
	+
	\mathbb E\left[\langle
	\Phi_j(X(t_{j-1}))
	-
	\Phi_j(X_h^{j-1})
	,\rho_j\rangle
	\right].
\end{equation}
	For the first term, since \(e_{j-1}\) is
	\(\mathcal F_{j-1}^h\)-measurable, the tower property of conditional expectation
	gives
	\[
	\mathbb E[\langle e_{j-1},\rho_j\rangle]
	=
	\mathbb E\left[
 \langle e_{j-1},\mathbb E\left[ \rho_j\mid\mathcal F_{j-1}^h\right]\rangle
	\right].
	\]
	Therefore, by the Cauchy--Schwarz inequality and
	\eqref{eq:local-residual-mean} in Lemma~\ref{lem:local-residual},
$$
	\begin{aligned}
		|\mathbb E\left[\langle e_{j-1},\rho_j\rangle\right]|
		&=
		\left|
		\mathbb E\left[
		\langle e_{j-1},\mathbb E\left[\rho_j\mid\mathcal F_{j-1}^h\right]\rangle
		\right]
		\right|
		\\
		&\le
		\|e_{j-1}\|_{L^2(\Omega)}
		\left\|
		\mathbb E\left[\rho_j\mid\mathcal F_{j-1}^h\right]
		\right\|_{L^2(\Omega)}
		\\
		&\le
		Ch_j^2\|e_{j-1}\|_{L^2(\Omega)}.
	\end{aligned}
$$
	Finally, by Young's inequality,
\begin{equation}\label{eq:product-ej-rho}
	|\mathbb E\left[\langle e_{j-1},\rho_j\rangle\right]|\leq Ch_j^2\|e_{j-1}\|_{L^2(\Omega)}
	\le
	Ch_j\|e_{j-1}\|_{L^2(\Omega)}^2
	+
	Ch_j^3
	=
	Ch_j\mathbb E\left[|e_{j-1}|^2\right]
	+
	Ch_j^3.
\end{equation}
	For the second term, by the Cauchy--Schwarz inequality, we have
	\[
	\begin{aligned}
		&
		\left|
		\mathbb E\left[\langle
		\Phi_j(X(t_{j-1}))
		-
		\Phi_j(X_h^{j-1})
		,\rho_j\rangle
		\right]
		\right|\le
		\left\|
		\Phi_j(X(t_{j-1}))
		-
		\Phi_j(X_h^{j-1})
		\right\|_{L^2(\Omega)}
		\|\rho_j\|_{L^2(\Omega)}.
	\end{aligned}
	\]
	Therefore, by \eqref{eq:conv-Phi-diff} and \eqref{eq:local-residual-L2} in Lemma \eqref{lem:local-residual}
\begin{equation}\label{eq:product-Psi-rho}
	\begin{aligned}
	\left|
	\mathbb E\left[\langle
	\Phi_j(X(t_{j-1}))
	-
	\Phi_j(X_h^{j-1})
	,\rho_j\rangle
	\right]
	\right|
	&\le
	Ch_j^{1/2}\|e_{j-1}\|_{L^2(\Omega)}\cdot Ch_j^{3/2}
	\\
	&=
	Ch_j^2\|e_{j-1}\|_{L^2(\Omega)}
	\\
	&\le
	Ch_j\mathbb E\left[|e_{j-1}|^2\right]+Ch_j^3,
\end{aligned}
\end{equation}
	where the last step follows from Young's inequality.
	Furthermore, Lemma~\ref{lem:local-residual} gives
\begin{equation}\label{eq:estimate-rho}
	\mathbb E\left[|\rho_j|^2\right]\le Ch_j^3.
\end{equation}
	Subscribing these estimates~\eqref{eq:product-eta-rho}-\eqref{eq:product-Psi-rho}, \eqref{eq:conv-eta-j} and \eqref{eq:estimate-rho} into \eqref{eq:expansion-ej}, we arrive at
	\[
	\mathbb E[|e_j|^2]
	\le
	(1+Ch_j)\mathbb E\left[|e_{j-1}|^2\right]
	+
	Ch_j^3 .
	\]
	Since \(e_0=0\), the discrete Gronwall lemma yields
	\[
	\mathbb E\left[|e_j|^2\right]
	\le
	C\sum_{i=1}^j h_i^3
	\le
	Ch^2\sum_{i=1}^j h_i
	\le
	Ch^2 .
	\]
	Taking square roots gives
	\[
	\|X(t_j)-X_h^j\|_{L^2(\Omega)}
	\le
	Ch,
	\qquad j=0,1,\ldots,N.
	\]
	Therefore,
	\[
	\max_{0\le j\le N}
	\|X(t_j)-X_h^j\|_{L^2(\Omega)}
	\le
	Ch.
	\]
	This proves \eqref{eq:main-convergence-final}.
\end{proof}
	\section{Numerical experiment}\label{sec:numerical-experiment}
	
	In this section, we will use three examples to validate our theoretical result. The first example is a one-dimensional SDE.
\subsection{Example 1}\label{EX1}

We first consider the scalar SDE
\begin{equation}
	\label{eq:test-sde-nondiff}
	dX(t)
	=
	\left(
	-X(t)^3+|X(t)|+a|\sin(\omega t)|
	\right)dt
	+
	\sigma\sin(X(t))\,dW(t),
	\qquad t\in[0,T],
\end{equation}
with initial value \(X(0)=1\). In this experiment, we choose
\[
T=1,\qquad
a=0.25,\qquad
\omega=20\pi,\qquad
\sigma=0.15.
\]
The drift and diffusion coefficients are therefore given by
\[
f(t,x)=-x^3+|x|+a|\sin(\omega t)|,
\qquad
g(x)=\sigma\sin(x).
\]

We briefly verify that this test equation satisfies the assumptions of
Theorem~\ref{thm:main-convergence}. The drift coefficient is not differentiable
with respect to \(x\) at \(x=0\). Moreover, for all \(x,y\in\mathbb R\),
\[
|f(t,x)-f(t,y)|
\le
C\left(1+|x|^2+|y|^2\right)|x-y|,
\]
and
\[
|f(t,x)|
\le
C(1+|x|^3).
\]
Thus the polynomial Lipschitz condition holds with exponent \(\kappa=2\).

The dissipative cubic term also implies a one-sided Lipschitz condition.
Indeed,
\[
\begin{aligned}
	(x-y)\bigl(f(t,x)-f(t,y)\bigr)
	&=
	-(x-y)^2(x^2+xy+y^2)
	+
	(x-y)\bigl(|x|-|y|\bigr)
	\\
	&\le
	|x-y|^2.
\end{aligned}
\]
Furthermore, for \(s,t\in[0,T]\),
\[
\begin{aligned}
	|f(t,x)-f(s,x)|
	&=
	a\left||\sin(\omega t)|-|\sin(\omega s)|\right|
	\\
	&\le
	a|\sin(\omega t)-\sin(\omega s)|
	\\
	&\le
	a\omega|t-s|
	\\
	&\le
	C(1+|x|^3)|t-s|^{1/2}.
\end{aligned}
\]
Hence the required temporal regularity condition is satisfied.

The diffusion coefficient \(g(x)=\sigma\sin(x)\) is globally Lipschitz, and its
derivative \(g'(x)=\sigma\cos(x)\) is bounded and globally Lipschitz. Moreover,
the Milstein coefficient
\[
g(x)g'(x)=\sigma^2\sin(x)\cos(x)
\]
is also globally Lipschitz. Since the diffusion coefficient is independent of
time, its temporal regularity condition is trivially satisfied. Therefore,
\eqref{eq:test-sde-nondiff} satisfies all the assumptions of
Theorem~\ref{thm:main-convergence}.

Since \(r=2\), we take the projection parameter
\[
\alpha=0.15<\frac{1}{2(\kappa+1)}=\frac16.
\]

To examine the strong convergence rate, we compute the root mean square error
at the terminal time \(T\),
\begin{equation}
	\label{eq:error-criterion}
	\mathrm{RMSE}(h)
	=
	\left(
	\frac1M
	\sum_{i=1}^M
	\left|X_h^{(i)}(T)-X_{\mathrm{ref}}^{(i)}(T)\right|^2
	\right)^{1/2}.
\end{equation}
Here \(M=2000\) sample paths are used. The reference solution
\(X_{\mathrm{ref}}\) is computed by the PRM method with the reference step size
\[
h_{\mathrm{ref}}=2^{-14},
\]
whereas the coarse step sizes are chosen as
\[
h=2^{-i},\qquad i=4,5,6,7,8.
\]
For each sample path, the coarse and reference solutions are driven by the same
Brownian path. The numerical results are shown in Fig.~\ref{E1}. The observed
convergence rate is \(1.206\), which is consistent with the theoretical
first-order convergence result.

\begin{figure}[!ht]
	\centering
	\includegraphics[width=100mm]{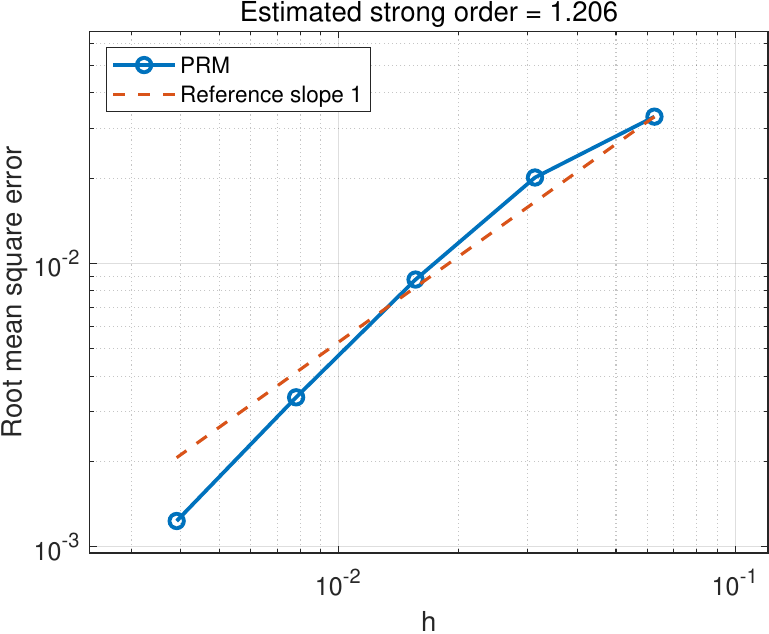}
	\caption{Strong convergence of the PRM method for the scalar SDE
		\eqref{eq:test-sde-nondiff}.}
	\label{E1}
\end{figure}

\subsection{Example 2}\label{Ex:2}

Although all three test equations satisfy the temporal regularity condition in
Assumption~\ref{ass:drift}, the role of this condition is different in the
present example. In \eqref{eq:test-sde-nondiff} and
\eqref{eq:fhn-example}, the time-dependent terms are Lipschitz continuous.
Hence, their \(1/2\)-H\"older estimates follow from the stronger Lipschitz
regularity. In contrast, the present example is designed to test the critical
\(1/2\)-H\"older setting, in which drift randomization is expected to improve
the convergence behavior of the classical left-point drift approximation.

We compare the proposed PRM method with the classical projected Milstein method
\cite{beyn2017stochastic}. To construct a drift coefficient which is smooth
with respect to the state variable but has critical temporal regularity, we use
a Takagi--Landsberg type function with Hurst parameter \(H=1/2\). Such functions
are classical examples of deterministic rough functions and have been studied
in connection with fractal regularity, \(p\)-variation, and pathwise calculus;
see, for example, \cite{mishura2019signed}. This choice is also consistent with
the motivation of drift-randomized Milstein methods, where randomized
quadrature is used to approximate drift integrals with H\"older-continuous time
dependence \cite{KruWu2019}.

Define
\begin{equation}
	\label{eq:eta}
	\eta(t)
	=
	\sum_{k=0}^{\infty}
	2^{-k/2}\operatorname{dist}(2^kt,\mathbb Z),
	\qquad
	\operatorname{dist}(u,\mathbb Z)
	=
	\min_{n\in\mathbb Z}|u-n|.
\end{equation}
We first verify the temporal regularity of \(\eta\). Since the mapping
\[
u\longmapsto\operatorname{dist}(u,\mathbb Z)
\]
is globally Lipschitz with Lipschitz constant one and satisfies
\[
0\le \operatorname{dist}(u,\mathbb Z)\le\frac12,
\]
we have, for all \(s,t\in[0,T]\),
\[
\left|
\operatorname{dist}(2^kt,\mathbb Z)
-
\operatorname{dist}(2^ks,\mathbb Z)
\right|
\le
\min\left\{2^k|t-s|,\frac12\right\}.
\]
Consequently,
\[
|\eta(t)-\eta(s)|
\le
\sum_{k=0}^{\infty}
2^{-k/2}
\min\left\{2^k|t-s|,\frac12\right\}.
\]
Set \(\delta=|t-s|\). The case \(\delta=0\) is trivial. For
\(0<\delta\le1\), choose \(n\in\mathbb N_0\) such that
\[
2^{-(n+1)}<\delta\le2^{-n}.
\]
It follows that
\[
\begin{aligned}
	|\eta(t)-\eta(s)|
	&\le
	\delta\sum_{k=0}^{n}2^{k/2}
	+
	\frac12\sum_{k=n+1}^{\infty}2^{-k/2}
	\\
	&\le
	C\delta 2^{n/2}
	+
	C2^{-n/2}
	\\
	&\le
	C\delta^{1/2}.
\end{aligned}
\]
Therefore,
\begin{equation}
	\label{eq:eta-holder}
	|\eta(t)-\eta(s)|
	\le
	C|t-s|^{1/2},
	\qquad s,t\in[0,T].
\end{equation}

Moreover, the exponent \(1/2\) is critical. Indeed, let \(t_n=2^{-n}\).
Since \(\eta(0)=0\), we have
\[
\begin{aligned}
	\eta(t_n)
	&=
	\sum_{k=0}^{n-1}
	2^{-k/2}\operatorname{dist}(2^{k-n},\mathbb Z)
	\\
	&=
	2^{-n}\sum_{k=0}^{n-1}2^{k/2}
	\ge
	2^{-(n+1)/2}.
\end{aligned}
\]
Hence, for every \(\nu>1/2\),
\[
\frac{|\eta(t_n)-\eta(0)|}{|t_n|^\nu}
\ge
2^{-1/2}2^{n(\nu-1/2)}
\longrightarrow\infty
\qquad\text{as }n\to\infty.
\]
Thus, \(\eta\) is not \(\nu\)-H\"older continuous for any
\(\nu>1/2\). In particular, it is not Lipschitz continuous.

We consider the scalar SDE
\begin{equation}
	\label{eq:test-sde-holder-time}
	dX(t)
	=
	\left[
	-X(t)^3+X(t)+a\eta(t)
	\right]dt
	+
	\sigma\sin(X(t))\,dW(t),
	\qquad t\in[0,T],
\end{equation}
with initial value \(X(0)=1\).

For the numerical implementation, the infinite series \(\eta\) is approximated
by its truncated version
\begin{equation}
	\label{eq:eta-K}
	\eta_K(t)
	=
	\sum_{k=0}^{K}
	2^{-k/2}\operatorname{dist}(2^kt,\mathbb Z).
\end{equation}
The truncation error satisfies
\[
\begin{aligned}
	\sup_{t\in[0,T]}|\eta(t)-\eta_K(t)|
	&\le
	\frac12\sum_{k=K+1}^{\infty}2^{-k/2}
	\\
	&=
	\frac{2^{-(K+1)/2}}
	{2(1-2^{-1/2})}.
\end{aligned}
\]
In the experiment, we choose
\[
T=1,\qquad
a=1,\qquad
\sigma=0.15,\qquad
K=36.
\]
For \(K=36\), the above uniform truncation bound is approximately
$
4.61\times10^{-6},
$
which is considerably smaller than the reference step size
\(h_{\mathrm{ref}}=2^{-14}\). Thus, the truncation error is negligible on the
temporal scales considered in the experiment.

We next verify that \eqref{eq:test-sde-holder-time} satisfies the assumptions of
Theorem~\ref{thm:main-convergence}. Its drift coefficient is
\[
f(t,x)=-x^3+x+a\eta(t).
\]
Since \(\eta\) is uniformly bounded on \([0,T]\), we have
\[
|f(t,x)-f(t,y)|
\le
C(1+|x|^2+|y|^2)|x-y|
\]
and
\[
|f(t,x)|
\le
C(1+|x|^3).
\]
Moreover,
\[
\begin{aligned}
	(x-y)\bigl(f(t,x)-f(t,y)\bigr)
	&=
	-(x-y)^2(x^2+xy+y^2)+(x-y)^2
	\\
	&\le
	|x-y|^2.
\end{aligned}
\]
Thus, the polynomial Lipschitz and one-sided Lipschitz conditions hold with
\(\kappa=2\).

By \eqref{eq:eta-holder},
\[
\begin{aligned}
	|f(t,x)-f(s,x)|
	&=
	a|\eta(t)-\eta(s)|
	\\
	&\le
	C|t-s|^{1/2}
	\\
	&\le
	C(1+|x|^{\kappa+1})|t-s|^{1/2}.
\end{aligned}
\]
Hence, the drift coefficient has precisely the temporal regularity required in
Assumption~\ref{ass:drift}. Unlike the time-dependent terms in \eqref{eq:test-sde-nondiff} and
\eqref{eq:fhn-example}, this temporal dependence is not Lipschitz continuous.

The diffusion coefficient \(g(x)=\sigma\sin(x)\) satisfies the same regularity
conditions as in \eqref{eq:test-sde-nondiff}. Therefore, all assumptions of
Theorem~\ref{thm:main-convergence} are satisfied. We choose
\[
\alpha=0.15
<
\frac{1}{2(\kappa+1)}
=
\frac16.
\]

Using the same error criterion \eqref{eq:error-criterion}, number of sample
paths, reference step size, and coarse step sizes as in the previous example,
we compare the proposed PRM method with the classical projected Milstein
method. In both methods, the function \(\eta\) is evaluated numerically through
the truncation \(\eta_{36}\) in \eqref{eq:eta-K}. The numerical results are
shown in Fig.~\ref{E2}. Over the tested step-size range, the PRM method exhibits
an approximately first-order convergence rate, whereas the classical projected
Milstein method has a lower empirical rate. This comparison illustrates the
advantage of randomized drift quadrature in the critical \(1/2\)-H\"older
temporal regularity regime.

\begin{figure}[!ht]
	\centering
	\includegraphics[width=100mm]{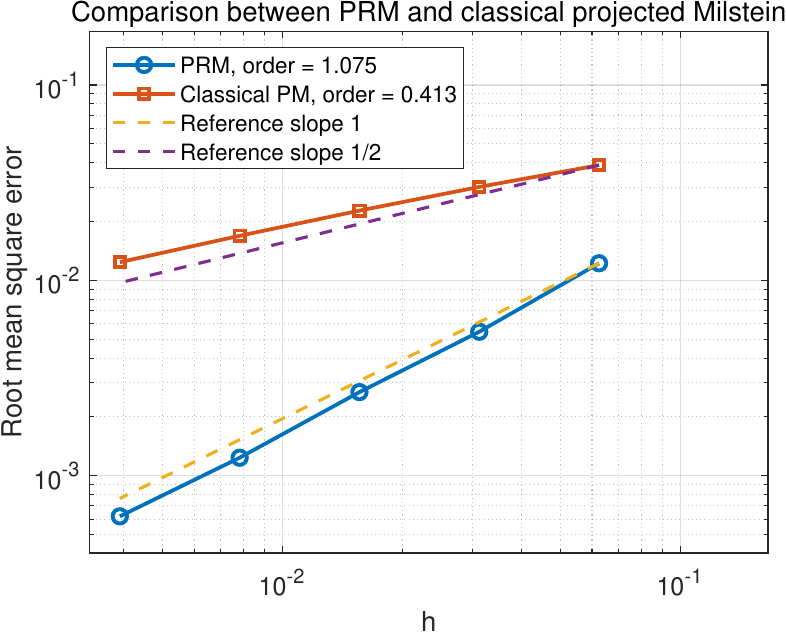}
	\caption{Strong convergence of the PRM and classical PM methods for the
		scalar SDE \eqref{eq:test-sde-holder-time}, where the Takagi--Landsberg
		function \(\eta\) is approximated by \(\eta_{36}\).}
	\label{E2}
\end{figure}

\subsection{Example 3}\label{Ex:3}

The third example is used to examine the applicability of the proposed PRM
method to a practically motivated non-differentiable drift coefficient. It is
motivated by FitzHugh--Nagumo type neuron models, which have also been used as
test problems in the numerical analysis of Milstein-type methods; see, for
example, \cite{biswas2026randomized}. The FitzHugh--Nagumo system is a classical
reduced model for excitable dynamics of nerve membranes
\cite{FitzHugh1961,Nagumo1962}, where the first component represents the
membrane potential and the second component represents a recovery variable.
Piecewise-linear and stochastic variants of this model have also been studied
in the literature; see, for example, \cite{SimpsonKuske2010}.

Motivated by these models, we consider a nonsmooth FitzHugh--Nagumo type system
with a threshold activation term. More precisely, we add the nonsmooth term
\((V-b)^+\) to the cubic FitzHugh--Nagumo drift, where
\[
(V-b)^+:=\max\{V-b,0\}.
\]
This term models an additional activation effect once the membrane potential
exceeds the threshold \(b\), and it is not differentiable at \(V=b\).
We consider
\begin{equation}
	\label{eq:fhn-example}
	\begin{aligned}
		dV(t)
		&=
		\left[
		-V(t)^3+V(t)-R(t)
		+\beta(V(t)-b)^+
		+a\sin(\omega t)
		\right]dt
		+\sigma\sin(V(t))\,dW(t),
		\\
		dR(t)
		&=
		\varepsilon
		\left(
		V(t)+\gamma-\lambda R(t)
		\right)dt,
		\qquad t\in[0,T].
	\end{aligned}
\end{equation}

In the numerical experiment, we take
\[
T=1,\qquad
X(0)=(V(0),R(0))^\top=(0.20,0)^\top,
\]
and choose
\[
\sigma=0.6,\qquad
\beta=0.5,\qquad
b=0.2,\qquad
a=0.5,\qquad
\omega=2\pi,
\]
\[
\varepsilon=0.08,\qquad
\gamma=0.5,\qquad
\lambda=1.
\]

We briefly verify that \eqref{eq:fhn-example} satisfies the assumptions of
Theorem~\ref{thm:main-convergence}. Let
\[
X=(V,R)^\top,
\qquad
Y=(\bar V,\bar R)^\top.
\]
The drift and diffusion coefficients are
\[
F(t,X)
=
\begin{pmatrix}
	-V^3+V-R+\beta(V-b)^++a\sin(\omega t)
	\\
	\varepsilon(V+\gamma-\lambda R)
\end{pmatrix},
\qquad
G(X)
=
\begin{pmatrix}
	\sigma\sin V\\
	0
\end{pmatrix}.
\]
The map \(V\mapsto(V-b)^+\) is globally Lipschitz but not differentiable at
\(V=b\). Consequently, \(F\) is not differentiable with respect to the state
variable on the hyperplane
\[
\{(V,R)\in\mathbb R^2:V=b\}.
\]

Using
\[
|V^3-\bar V^3|
\le
C(1+|V|^2+|\bar V|^2)|V-\bar V|
\]
and
\[
|(V-b)^+-(\bar V-b)^+|
\le
|V-\bar V|,
\]
we obtain
\[
|F(t,X)-F(t,Y)|
\le
C(1+|X|^2+|Y|^2)|X-Y|,
\]
and
\[
|F(t,X)|\le C(1+|X|^3).
\]
Thus the polynomial Lipschitz condition holds with \(r=2\).

The one-sided Lipschitz condition follows from
\[
\begin{aligned}
	&\left\langle X-Y,F(t,X)-F(t,Y)\right\rangle
	\\
	&=
	-(V-\bar V)^2(V^2+V\bar V+\bar V^2)
	+(V-\bar V)^2
	\\
	&\quad
	+\beta(V-\bar V)
	\left((V-b)^+-(\bar V-b)^+\right)
	\\
	&\quad
	+(\varepsilon-1)(V-\bar V)(R-\bar R)
	-\varepsilon\lambda(R-\bar R)^2.
\end{aligned}
\]
Since
\[
0
\le
(V-\bar V)
\left((V-b)^+-(\bar V-b)^+\right)
\le
|V-\bar V|^2,
\]
Young's inequality yields
\[
\left\langle X-Y,F(t,X)-F(t,Y)\right\rangle
\le
C|X-Y|^2.
\]

The temporal regularity follows from
\[
\begin{aligned}
	|F(t,X)-F(s,X)|
	&=
	a|\sin(\omega t)-\sin(\omega s)|
	\\
	&\le
	a\omega|t-s|
	\\
	&\le
	C(1+|X|^3)|t-s|^{1/2}.
\end{aligned}
\]
Finally, \(G\) is globally Lipschitz, and its Jacobian is
\[
\partial_XG(X)
=
\begin{pmatrix}
	\sigma\cos V&0\\
	0&0
\end{pmatrix},
\]
which is bounded and globally Lipschitz. The corresponding Milstein coefficient
is
\[
\partial_XG(X)G(X)
=
\begin{pmatrix}
	\sigma^2\sin V\cos V\\
	0
\end{pmatrix},
\]
which is also globally Lipschitz. Since \(G\) is independent of time, its
temporal regularity condition is trivially satisfied. Therefore,
\eqref{eq:fhn-example} satisfies all the assumptions of
Theorem~\ref{thm:main-convergence}.

The projection parameter is chosen as
\[
\alpha=0.15<\frac{1}{2(\kappa+1)}=\frac16.
\]
We use \(M=2500\) sample paths, the reference step size
\[
h_{\mathrm{ref}}=2^{-15},
\]
and the coarse step sizes
\[
h=2^{-i},\qquad i=4,5,6,7,8.
\]
Using the error criterion \eqref{eq:error-criterion}, we compute the RMSE at the
terminal time \(T\). The numerical results are shown in Fig.~\ref{E3}. The
observed convergence rate is \(1.073\), which is consistent with the strong
convergence order proved in Theorem~\ref{thm:main-convergence}. This example
illustrates that the proposed PRM method can be effectively applied to a
practically motivated stochastic FitzHugh--Nagumo type system with a
non-differentiable and super-linearly growing drift coefficient.

\begin{figure}[!ht]
	\centering
	\includegraphics[width=100mm]{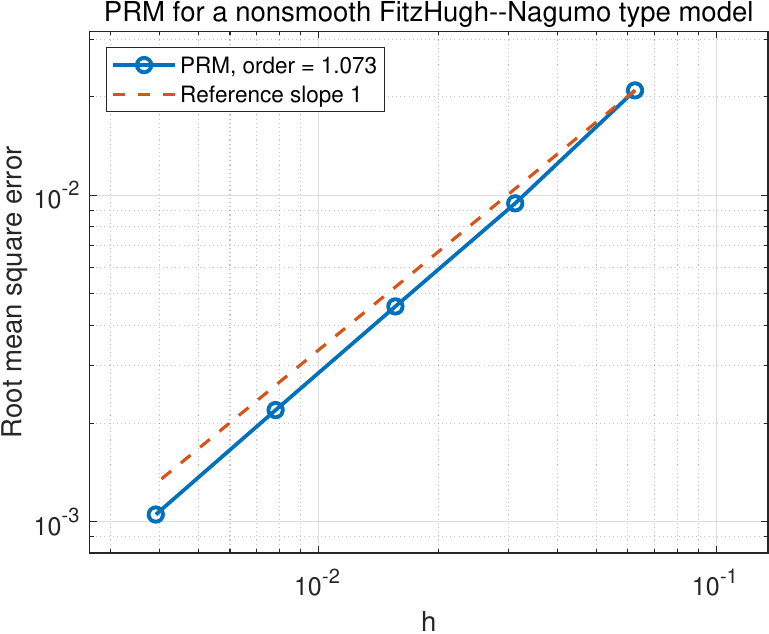}
	\caption{Strong convergence of the PRM method for the nonsmooth
		FitzHugh--Nagumo type system \eqref{eq:fhn-example}.}
	\label{E3}
\end{figure}
	\section*{Conclusion and Discussion}

In this paper, we proposed a projected drift-randomized Milstein method for
stochastic differential equations with non-differentiable and super-linearly
growing drift coefficients. The method combines the drift randomization technique
with a projection applied only to the drift coefficient, which allows us to
control the super-linear growth while preserving the applicability to
non-differentiable drifts.

Under suitable one-sided Lipschitz, polynomial growth, and regularity
assumptions, we proved that the proposed method converges strongly with order
one uniformly over the temporal grid points in the \(L^2\)-sense. The analysis
is based on uniform moment estimates, a one-step mean-square stability argument,
and local residual estimates for the randomized drift and Milstein diffusion
terms.

To the best of our knowledge, this is the first randomized Milstein-type method
that achieves first-order strong convergence for SDEs whose drift coefficients
are both non-differentiable and super-linearly growing. This extends the
applicability of randomized Milstein methods beyond the globally Lipschitz
setting. Moreover, compared with existing Milstein-type methods for SDEs with
super-linearly growing drift coefficients, the proposed method does not require
differentiability of the drift coefficient. Motivated by related works on
randomized numerical methods for more general stochastic systems
\cite{biswas2024explicit,morkisz2021randomized,przybylowicz2024randomized}, it
would be interesting to extend the present approach to other classes of SDEs
with super-linearly growing coefficients, such as McKean--Vlasov SDEs and
jump--diffusion SDEs. These extensions are left for future work.
\section*{Declaration on the Use of Generative AI}

Generative artificial intelligence tools were used in the preparation of this
manuscript for language polishing, improving the clarity of presentation,
assisting with some routine mathematical derivations, and suggesting the
construction and implementation of numerical examples. All mathematical
arguments, assumptions, proofs, numerical algorithms, and computational results
were independently checked and verified by the author. The author takes full
responsibility for the content of the manuscript.
\bibliography{ref}
\bibliographystyle{unsrt}
\end{document}